\documentclass[a4paper,twoside,10pt]{amsart}

\usepackage{amsthm,amsmath,amssymb,amscd}
\usepackage[greek, french, english]{babel}			
\usepackage{graphicx}
\usepackage[all]{xy}
\usepackage{enumitem}
\usepackage{comment} 
\usepackage{color}
\usepackage[T1]{fontenc}
\usepackage[colorlinks=true,citecolor=blue, urlcolor=blue,linkcolor=blue]{hyperref}

\setlist[itemize,enumerate]{leftmargin=*}

\newtheorem{theorem}{Theorem}[section]
\newtheorem{proposition}[theorem]{Proposition}
\newtheorem{lemma}[theorem]{Lemma}
\newtheorem{corollary}[theorem]{Corollary}

\newtheorem{innercustomthm}{Theorem}
	\newenvironment{customthm}[1]
  {\renewcommand\theinnercustomthm{#1}\innercustomthm}
  {\endinnercustomthm}

\theoremstyle{definition}
\newtheorem{definition}[theorem]{Definition}

\theoremstyle{remark}
\newtheorem{remark}[theorem]{Remark}
\newtheorem{Remark}[theorem]{Remarks}

\newlength{\szer}

\def\N{{\mathbf N}}

\def\Q{{\mathbf Q}}

\def\Z{{\mathbf Z}}
\def\Z{\mathbf Z}

\def\si{\sigma_{\infty}}

\def\int{\mathrm{int}}

\def\elem(#1,#2){  \{ \frac{#1}{\overline{\ #2\ }}\}  }

\newcommand{\set}[1]{\left\{#1\right\}}
\newcommand{\RZ}[1]{\textnormal{RZ}#1}
\makeatletter
\newcommand*\rel@kern[1]{\kern#1\dimexpr\macc@kerna}
\newcommand*\widebar[1]{%
  \begingroup
  \def\mathaccent##1##2{%
    \rel@kern{0.8}%
    \overline{\rel@kern{-0.8}\macc@nucleus\rel@kern{0.2}}%
    \rel@kern{-0.2}%
  }%
  \macc@depth\@ne
  \let\math@bgroup\@empty \let\math@egroup\macc@set@skewchar
  \mathsurround\z@ \frozen@everymath{\mathgroup\macc@group\relax}%
  \macc@set@skewchar\relax
  \let\mathaccentV\macc@nested@a
  \macc@nested@a\relax111{#1}%
  \endgroup
}
\makeatother

\begin{document}
\setcounter{tocdepth}{2}
\title{VALUATIONS AND HENSELIZATION}
\author{Ana Bel\'en de Felipe and Bernard Teissier}
\address{ABdF: BCAM - Basque Center for Applied Mathematics, Mazarredo 14, 
E-48009 Bilbao, Basque Country--Spain.}
\email{adefelipe@bcamath.org}
\address{BT: IMJ-PRG, CNRS, Universit\'e Paris Diderot, Institut de Math\'ematiques de Jussieu-Paris Rive Gauche, B\^at. Sophie Germain, Place Aur\'elie Nemours, F-75013, Paris, France.}
\email{bernard.teissier@imj-prg.fr}
\dedicatory{To the memory of Jean-Pierre Lafon}
\thanks{The first author is supported by ERCEA Consolidator Grant 615655-NMST and also by the Basque Government through the BERC 2018-2021 program, and by Spanish Ministry of Economy and Competitiveness MINECO: BCAM Severo Ochoa excellence accreditation SEV-2017-0718 and MTM2016-80659-P. Part of this work was carried when she was a member of the Institute of Mathematics of the University of Barcelona.}
\newpage

\pagestyle{myheadings}
\markboth{\rm A. de Felipe, B. Teissier}{\rm Extension of valuations}

\begin{abstract}We study the extension of valuations centered in a local domain to its henselization. We prove that a valuation $\nu$ centered in a local domain $R$ uniquely determines a minimal prime $H(\nu)$ of the henselization $R^h$ of $R$ and an extension of $\nu$ centered in $R^h/H(\nu)$, which has the same value group as $\nu$. Using the integrality and functoriality of henselization, this is equivalent to the fact that the henselization of a valuation ring is a valuation ring with the same value group, which is a fundamental result in the theory of valued fields. We present here a more constructive approach to this result and some consequences of this approach. Our method, which assumes neither that $R$ is noetherian nor that it is integrally closed,  is to reduce the problem to the extension of the valuation to a quotient of a standard \'etale local $R$-algebra and in that situation to draw valuative consequences from the observation that the Newton--Hensel algorithm for constructing roots of polynomials produces sequences that are, for any valuation centered in $R$, pseudo--convergent in the sense of Ostrowski.\par\noindent We then apply this method to the study of the approximation of elements of the henselization of a valued field by elements of the field and give a characterization of the henselian property of a local domain $(R,m_R)$ in terms of the limits of certain pseudo--convergent sequences of elements of $m_R$ for a valuation centered in it. Another consequence of our work is to establish in full generality a bijective correspondence between the minimal primes of the henselization of a local domain $R$ and the connected components of the Riemann--Zariski space of valuations centered in $R$.\end{abstract}

\maketitle
\section{Introduction}
Henselization plays an important role in the theory of valued fields, in particular because the valuation of a henselian valued field extends uniquely to any algebraic extension and because maximal valued fields are henselian (see section \ref{sec4} for a refinement of this). Henselization of local rings is also a fundamental tool in algebraic geometry, in particular because of its relation with the implicit function theorem and the detection of local analytic branches of an excellent scheme, but also because of the role it plays in the study of extensions of a valuation of an excellent local domain $(R,m_R)$ to its $m_R$-adic completion, and thus in some approaches to local uniformization in arbitrary characteristic. In this paper we study the henselization of local domains with a view to applications to algebraic geometry but with methods close to those of the theory of valued fields. More precisely, we study the Newton--Hensel algorithm which is the origin of the henselian property from a valuative viewpoint based on the observation that when applied to a polynomial $F(X)\in R[X]$ defining a standard \'etale extension this algorithm always produces a pseudo--convergent sequence in the sense of Ostrowski of elements of the maximal ideal of $R$.   

\subsection*{Extension of valuations to the henselization} A valuation on a local domain $R$ is a valuation $\nu\colon K^*\to\Phi$ of the fraction field $K$ of $R$ with values in an abelian totally ordered group $\Phi$ such that the \textit{value semigroup} $\nu(R\setminus\{0\})$ of the valuation is contained in $\Phi_{\geq 0}$. One usually adds an element $\infty$ larger than all elements of $\Phi$ and sets $\nu(0)=\infty$. If we denote by $R_\nu\subset K$ the valuation ring of $\nu$, then a valuation $\nu$ on $R$ corresponds to an inclusion $R\subseteq R_\nu$. The valuation is \textit{centered in} $R$ if the maximal ideal $m_R$ of $R$ is the ideal of elements with value in $\Phi_{>0}\cup\{\infty\}$. This means that $R_\nu$ \textit{dominates} $R$ in the sense that $R\subseteq R_\nu$ and, denoting by $m_\nu$ the maximal ideal of $R_\nu$, we have the equality $m_\nu\cap R=m_R$. If $K\hookrightarrow L$ is a field extension, then a valuation $\tilde\nu$ of $L$ is an \emph{extension} of $\nu$ if $R_{\tilde\nu}\cap K=R_\nu$ where $R_{\tilde\nu}$ denotes the valuation ring of $\tilde\nu$.\par

The main result of this article is the following:

\begin{customthm}{1}\label{mainthm}
Let $R$ be a local domain and let $R^h$ be its henselization. If $\nu$ is a valuation centered in $R$, then:
\begin{enumerate}
    \item\label{hext} There exists a unique prime ideal $H(\nu)$ of $R^h$ lying over the zero ideal of $R$ such that $\nu$ extends to a valuation $\tilde\nu$ centered in $R^h/H(\nu)$ through the inclusion $R\subset R^h/H(\nu)$. In addition, the ideal $H(\nu)$ is a minimal prime and the extension $\tilde\nu$ is unique.
    \item\label{grouphext} With the notation of \eqref{hext}, the valuations $\nu$ and $\tilde\nu$ have the same value group.
\end{enumerate}
\end{customthm}

Remembering that a semivaluation centered in a local ring $T$ and supported at a prime ideal $\mathfrak p$ of $T$ means a valuation centered in the local domain $T/\mathfrak p$, we can paraphrase this as:\par\noindent
\textit{Any valuation centered in a local domain has a unique extension as a semivaluation centered in its henselization, which has the same value group and is supported at a minimal prime.}\par\medskip

Recall that the henselization $R\to R^h$ factorizes uniquely local morphisms from $R$ to henselian local rings; it is unique up to unique isomorphism of $R$-algebras. We fix one.\par
A large part of Theorem~\ref{mainthm} was proved in \cite{HOST} under the additional assumption that $R$ is (quasi)-excellent and thus noetherian. The assumption of quasi--excellence is used in particular to ensure finiteness properties for the normalizations of the ring $R$ and its quotients. We establish the result in full generality and, in contrast to the proof given in \cite{HOST}, or the proof using the fact that the henselization of a valuation ring is a valuation ring (see Remark \ref{hh}), our proof has the advantage of being constructive.

\subsection*{Motivation} Theorem~\ref{mainthm} plays an important role in the classification of extensions of valuations of a quasi-excellent local domain to its completion according to the method of \cite{HOST}. One goal of this classification is to prove the following analogue of the second part of Theorem~\ref{mainthm}: \textit{There exists an extension of a valuation of a  quasi-excellent local domain $R$ to a semivaluation {\rm with the same value group} of its completion $\hat R$}. This yet unproven statement seems crucial for proofs of local uniformization in positive characteristic. Such an extension of the valuation allows one to use the advantages of completeness for intermediate steps without losing the algebraicity of the uniformizing modifications in the end. See Conjecture 1.1 of \cite{HOST} which sets in the framework of \cite{HOST} and makes more precise a conjecture of the second author going back to 2003 (see \cite[*proposition* 5.19]{T}).\par 

Theorem~\ref{mainthm}.\eqref{hext} is also an essential ingredient in the proof by the first author that, up to homeomorphism in the Zariski topology, the space of valuations centered in a non singular point of an algebraic variety over an algebraically closed field depends only on the dimension (see \cite[Theorem 3.5]{F}). In this direction, Theorem~\ref{mainthm}.\eqref{hext} allows us to establish in full generality a bijective correspondence between the minimal primes of the henselization of a local domain $R$ and the connected components of the Riemann--Zariski space of valuations centered in $R$ (see Corollary~\ref{homeoRZ}).

\subsection*{Strategy of the proof} The henselization of a local domain $R$ is the limit of the inductive system of its \textit{Nagata extensions} (also known as standard \'etale extensions) $S=R[X]/(F(X))_{(m_R,X)}$ where $F(X)=X^n+\cdots+a_{n-1}X-a_n\in R[X]$  is a \textit{Nagata polynomial}: $a_n\in m_R,\ a_{n-1}\notin m_R$. It suffices therefore to study extensions of valuations centered in $R$ to its Nagata extensions. Our approach is to do this from the viewpoint of Ostrowski's pseudo--convergent sequences, which gives a simple description of the minimal prime and an explicit construction of the extended valuation implying immediately that the value group does not change. We give more details in what follows.\par

Let us keep the notation introduced above. Let $S$ be a Nagata extension of $R$ defined by a Nagata polynomial $F(X)\in R[X]$. Any extension of $\nu$ to the algebraic closure $\widebar K$ of $K$ determines a root $\si\in\widebar K$ of $F(X)$ that is a limit in the sense of Ostrowski of a pseudo--convergent sequence of elements $\sigma_i\in m_R$, $i\geq1$, which is attached to $F(X)$ by Newton's method. The minimal polynomial over $K$ of all these distinguished roots $\si$ coincide, producing a minimal prime $H_S(\nu)$ of $S$ (which is the trace of the minimal prime $H(\nu)$ of $R^h$). The ideal $H_S(\nu)$ and the extension of $\nu$ to a valuation centered in the local domain $S/H_S(\nu)$ are unique, and Theorem~\ref{mainthm}.\eqref{hext} follows by passage to the inductive limit.\par

We prove Theorem~\ref{mainthm}.\eqref{grouphext} by proving that the value group $\Phi$ of $\nu$ does not change after extension to $S/H_S(\nu)$. To do this, we fix a presentation of the previous quotient as a local $R$--algebra $R[\si]_{(m_R,\si)}\subset\widebar K$ and investigate the way in which $\nu$ determines the value of the extended valuation $\tilde\nu$ on each element $h(\si)$ with $h(X)\in R[X]$. We describe the behavior of the valuations $\nu(h(\sigma_i))$, $i\geq1$. Indeed, if these valuations form an eventually constant sequence, then their stationary value is $\tilde\nu(h(\si))$; and otherwise, they are cofinal in a certain convex subgroup of $\Phi$. Except in some particular cases (for instance, if the valuation $\nu$ is of rank one, in which case the cofinality implies that $h(\si)=0$), this is not sufficient to obtain the desired result.\par

In order to compute $\tilde\nu(h(\sigma_\infty))$ in general, we present the domain $R$ as the inductive limit of its subdomains $A_0$ which are essentially of finite type over the prime ring, and for which the restriction $\nu_0$ of $\nu$ is of finite rank by Abhyankar's inequality. With the notation of Theorem~\ref{mainthm}, the local domain $R^h/H(\nu)$ is then the inductive limit of the $A_0^h/H(\nu_0)$ and each $\tilde\nu_0$ is obtained by restriction of $\tilde\nu$. Hence we can choose $A_0$ containing the coefficients of $F(X)$ and $h(X)$, and compute $\tilde\nu(h(\si))$ as the $\tilde\nu_0$--value of $h(\si)$ seen as element of $A_0[\si]_{(m_{A_0},\si)}\subset R[\si]_{(m_R,\si)}$. After an iterative procedure (the main point here is that $\nu_0$ has finite rank), we determine a finite extension of the fraction field of $A_0$, a unique extension $\nu_\ell$ of $\nu_0$ to this new field, which has the same value group, and a pseudo--convergent sequence $(\chi_i^{(\ell)})_{i\geq1}$ for $\nu_\ell$ such that $\tilde\nu_0(h(\si))$ equals $\nu_\ell(h(\chi_i^{(\ell)}))$ for all $i$ large enough.\par

\subsection*{Applications of our method} The subdomains of $R$ which are essentially of finite type over the prime ring are excellent, and we could have applied directly this argument to reduce the general case to the case treated in \cite{HOST}. However, we think that our method is much more informative. For example, since it applies to valuation rings we use it to prove the result \cite[Theorem 1.1]{Ku} of F--V. Kuhlmann on the approximation of elements of the henselization of a valued field by elements of the field and a general characterization of the henselian property in terms of pseudo--convergent sequences (see Theorem~\ref{ap} and Proposition~\ref{valcrithen}, respectively). We hope that our method can also be used to study the changes in the value \textit{semigroup} which can take place when passing to the henselization, even for regular local rings, as discovered by Cutkosky in \cite[Theorem~1.5]{C}.\par

\subsection*{Organization of the paper} 

In Sections~\ref{sec2} and~\ref{sec3} we prove Theorem~\ref{mainthm}.\eqref{hext} and Theorem~\ref{mainthm}.\eqref{grouphext}, respectively. We avoid using the approximations of $R$ by its noetherian subrings as much as we can because some general results such as Lemma~\ref{rat} may be of independent interest. The last two sections are applications of our approach. In Subsection~\ref{RZ} we study the decomposition into connected components of the Riemann--Zariski space of valuations centered in a local domain, and in Subsection~\ref{sec5} we revisit the result of F--V. Kuhlmann mentioned above. In Section~\ref{sec4} we propose a characterization of the henselian property of rings in terms of pseudo--convergent sequences.

\subsection*{Aknowledgements} We are grateful to Franz--Viktor Kuhlmann and Laurent Moret-Bailly for useful suggestions, and to the referee for his work and his useful suggestions.


\section{Nagata extensions and Newton-Hensel approximations from a valuative viewpoint}\label{sec2}

In this section we assume that $R$ is a local domain and $\nu$ a valuation centered in $R$, and we study from a valuative viewpoint the process of henselian approximation. We denote by $m_R$ the maximal ideal of $R$ and by $K$ its fraction field.\par

We start with the following result of J.-P. Lafon:\par\noindent
\begin{proposition}{\rm(\cite[Proposition 6]{L})}\label{standard} Let $R$ be a local ring with maximal ideal $m_R$ and $S$ a local $R$-algebra with maximal ideal $m_S$. The following assertions are equivalent: 
\begin{enumerate}
\item $S$ is a localization of a finite $R$-algebra and is flat over $R$, and $S/m_RS =R/m_R=S/m_S$.
\item $S$ is of the form $(R[X]/(F(X)))_{\mathcal N}$ where F(X) is a unitary polynomial of the form
\begin{eqnarray*} &&X^n + a_1X^{n-1} +\cdots +a_{n-1}X - a_n,\\
&&\qquad\text{ where } a_i\in R\text{ for }  1\leq i\leq n\text{ and }  a_n\in m_R,\ a_{n-1}\notin m_R,\end{eqnarray*}
and $\mathcal N$ is the maximal ideal of $R[X]/(F(X))$ containing the class $x$ of $X$ modulo $F(X)$, which is the image of the maximal ideal $(m_R,X)$ of $R[X]$.
\item $S$ is a localization of a finite $R$-algebra, and for every local subalgebra $R_0$ dominated by $R$ which is essentially of finite type over $\Z$ and contains the coefficients of $F(X)$ so that $F(X)\in R_0[X]$, the natural map $R_0\to S_0=(R_0[X]/(F(X)))_{\mathcal N_0}$, where ${\mathcal N_0}$ is the image of $(m_{R_0},X)$, induces an isomorphism of the completions.
\end{enumerate}
\end{proposition}

Lafon calls such extensions $R\to S$ \textit{Nagata extensions}\footnote{In \cite{N2}, Nagata calls them \textit{quasi-decompositional}.}; they are also called standard \'etale extensions of $R$ or, assuming that $R$ is noetherian, \'etale $R$-algebras quasi-isomorphic to $R$; see also \cite[\S~18]{EGA}. For brevity we shall call polynomials $F(X)$ with coefficients in a local ring satisfying the conditions of Proposition~\ref{standard} \textit{Nagata polynomials}\footnote{In \cite{L-M} they are called "polynomials satisfying the conditions of the implicit function theorem".}. We adopt the convention that the constant term of a Nagata polynomial has a minus sign.\par 

Morphisms of Nagata extensions of $R$ are local morphisms of local $R$-algebras. A morphism from a Nagata extension $S$ to another one $S'$ exists if and only if there is an element $\xi'$ in the maximal ideal of $S'$ such that $F(\xi')=0$. There exists at most one such morphism, determined by sending the image $x\in S$ of $X$ to $\xi'\in S'$ and then, by Proposition~\ref{standard}, $S'$ is a Nagata extension of $S$. Lafon proves that Nagata extensions of $R$ form an inductive system and (\cite[Th\'eor\`eme 2]{L}) that the henselization $R^h$ of $R$ is the inductive limit of its Nagata extensions. In particular it has the same residue field as $R$.\par\noindent
\begin{remark}\label{changeofvariable}
Keeping the notations of Proposition~\ref{standard}, note that if $a_n=0$, then $S$ is isomorphic to $R$. The extension is also trivial when $n=1$. Note also that given any element $\alpha\in m_R$, the polynomial $F_\alpha(X')=F(X'+\alpha)\in R[X']$ with $X'=X-\alpha$ satisfies the same conditions as $F(X)$. Indeed, $F_\alpha(0)=F(\alpha)\in m_R$; and the coefficient of $X'$ in $F_\alpha(X')$ is $F'(\alpha)$, which is not in $m_R$ since $F'(0)$ is not and $\alpha\in m_R$. Moreover, $F_\alpha(X')$ defines the same extension, that is, $S$ is isomorphic to $S_\alpha=(R[X']/(F_\alpha(X')))_\mathcal {N'}$. This implies that the Nagata extension defined by the Nagata polynomial $F(X)$ is trivial if and only if $F(X)$ has a zero in the maximal ideal of $R$.
\end{remark}

As a consequence of the following result, we may assume in the definition of a Nagata extension that the polynomial $F(X)$ is irreducible in $R[X]$. An equivalent statement is found in \cite[Chap.13, Proposition 13.15]{L-M} for the case where $R$ is integrally closed.\par\noindent
\begin{lemma}\label{fac} 
Let $R$ be a local domain and let $F(X)\in R[X]$ be a Nagata polynomial. Let $F(X)=G(X)Q(X)$ be a factorization in $R[X]$, where up to multiplication by a unit of $R$ we write 
\begin{eqnarray*}
G(X)=X^s+\cdots+g_{s-1}X+g_s;\qquad
Q(X)=X^t+\cdots +q_{t-1}X-q_t.
\end{eqnarray*}
Then, one of the two polynomials $G(X),\ Q(X)$ must be a Nagata polynomial. It is the factor whose constant term is in $m_R$. If it is $Q(X)$, then $G(X)\notin (m_R, X)$.
\end{lemma}

\begin{proof}
Let us consider the linear part of $F(X)$ as it is written in Proposition~\ref{standard}; we have $-a_n=-q_tg_s\in m_R,\ a_{n-1}=g_sq_{t-1}-g_{s-1}q_t\notin m_R$. Since $a_{n-1}\notin m_R$ it is impossible for both $g_s$ and $q_t$ to be in $m_R$, but one of them must be since $m_R$ is prime. Let us say that $q_t\in m_R$ and $g_s\notin m_R$ so that $G(X)\notin (m_R, X)$. Then by the second equality we have $q_{t-1}\notin m_R$ so that $Q(X)$ is a Nagata polynomial.\end{proof}

The next lemma presents Newton's method in the way we are going to use it:\par\noindent
\begin{lemma}\label{app} Let $F(X)=X^n + a_1X^{n-1} +\cdots +a_{n-1}X - a_n\in R[X]$ be a Nagata polynomial and note that as an element of $R[X]$, the polynomial $F(X)$ is the same as \[F^{(1)}(X_1)=F\left(X_1-\frac{F(0)}{F'(0)}\right)=F\left(X_1+\frac{a_n}{a_{n-1}}\right)\]
since $X\mapsto X_1+\frac{a_n}{a_{n-1}}$ is a change of variable in $R[X]$. Substituting $X_1-\frac{a_n}{a_{n-1}}$ to $X$ in $F(X)$, write the result $F^{(1)}(X_1)=X_1^n + a_1^{(1)}X_1^{n-1} +\cdots +a_{n-1}^{(1)}X_1 - a_n^{(1)}$. Then we have:
\begin{enumerate}
\item The polynomial $F^{(1)}(X_1)\in R[X_1]$ is a Nagata polynomial.
\item The coefficient $a^{(1)}_i$ is congruent to $a_i$ modulo $\frac{a_n}{a_{n-1}}$.
\item  $F^{(1)}(0)=-a^{(1)}_n\in a_n^2R$.
\item Let $R\to S$ be the Nagata extension defined by $F(X)$. Denoting by $x$, $x_1$ the images in $S$ of $X,X_1$, we have $x_1\in x^2S$. In particular, if $\tilde\nu$ is any semivaluation on $S$ extending the valuation $\nu$ on $R$, the inequality $\tilde\nu (x_1)\geq 2\tilde\nu (x)$ holds.
\end{enumerate}
\end{lemma}

\begin{proof}
The first statement is what was remarked above, and the proof of the next two is a direct computation. The last one follows from the fact that modulo $F(X)$, the element $X_1=X-\frac{a_n}{a_{n-1}}$ is a multiple of $X^2$.
\end{proof}

As a consequence, starting from a Nagata polynomial $F(X)\in R[X]$, we can iterate the construction just described to produce:
\begin{itemize}
\item A sequence of generators $X_i:=X_{i-1}+\frac{F^{(i-1)}(0)}{(F^{(i-1)})'(0)}$ for the polynomial ring $R[X]$, with $X_0=X$.
\item Polynomials $F^{(i)}(X_i):=F^{(i-1)}\left(X_i-\frac{F^{(i-1)}(0)}{(F^{(i-1)})'(0)}\right)\in R[X_i]$, with $F^{(0)}(X)=F(X)$. 
\end{itemize}

\begin{definition}\label{def:sequences}
Let $\nu$ be a valuation centered in a local domain $R$ and let $F(X)\in R[X]$ be a Nagata polynomial. Keep the previous notations. We define the following elements of $m_R$: 
\begin{equation*}
\begin{aligned}
\delta_k&:=\frac{a_n^{(k)}}{a_{n-1}^{(k)}}=-\frac{F^{(k)}(0)}{(F^{(k)})'(0)},\text{ for }k\geq0.\\
\sigma_i&:=\sum_{k=0}^{i-1}\delta_k,\text{ for }i\geq1.
\end{aligned}
\end{equation*}
We say that $(\delta_i)_{i\in\N}$ and $(\sigma_i)_{i\geq1}$ are the \emph{Newton sequence of values} and the \emph{sequence of partial sums }attached to $F(X)$, respectively.
\end{definition}

The polynomials $(F^{(i)}(X_i))_{i\in \N}$ all define the same Nagata extension of $R$. If at some step $i\geq0$ we find $F^{(i)}(0)=0$, this implies that $F(X)$ defines a trivial extension, so we may assume that this does not happen and we shall do so.\par 

By construction, we have $X=X_i+\sigma_i$ and $x_{i+1}=x_i-\delta_i$. We verify by induction that 
$F^{(i)}(X_i)=F\left(X_i+\sigma_i\right)$ for $i\geq1$. Setting $X_i=0$ in this identity, we can read the definition of $\delta_i$ as given by the equality $F'(\sigma_i)\delta_i=-F(\sigma_i)$. Observe that for all $i\geq1$, $F(\sigma_i)\neq0$ because the Nagata extension is not trivial, and $\nu(\delta_i)=\nu(F(\sigma_i))$.\par\noindent 
\begin{remark}\label{hensel} Assuming for a moment that $R$ is complete and separated for the $m_R$-adic topology, Lemma~\ref{app} tells us in particular that the images in $S$ of the elements $X_i$ converge to $x_\infty=0$ while the polynomials $F^{(i)}(X_i)$ converge to a polynomial $F^{(\infty)}(X_\infty)$ without constant term because $a^{(i)}_n\in m_R^{2^{i-1}}$. Therefore $x_\infty$ is a root of $F^{(\infty)}(X_\infty)$, which is simple since $a^{(\infty)}_{n-1}\notin m_R$. Since $x_\infty=x-\sum_{k=0}^\infty\delta_k$ and $F^{(\infty)}(X_\infty)=F(X)$ this tells us that $\sum_{k=0}^\infty\delta_k$ is a simple root of $F(X)$, which is contained in the maximal ideal $m_R$ of $R$. Since our assumption on $F(X)$ is equivalent to the statement that the image of $F(X)$ in $k[X]$, where $k=R/m_R$,  has $0$ as a simple root, this is indeed a version of Hensel's lemma.\end{remark}

We stress the fact that by our assumption that the Nagata extension is not trivial we have $\delta_0\in m_R\setminus\{0\}$ and $\delta_{i+1}$ is a non zero multiple of $\delta_i^2$ for any $i\geq0$, so that we expect to have a root of $F(X)$ which is represented as a sum $\sum_{k=0}^\infty\delta_k$ of elements of strictly increasing valuations.\par

In general, the $m_R$-adic topology may not be separated, in fact one can have $\bigcap_{n\in\N}m_R^n=m_R$, but the partial sums $(\sigma_i)_{i\geq1}$ form a pseudo--Cauchy, or pseudo--convergent sequence in the sense of Ostrowski \cite[Teil III, \S~11]{O} for the valuation $\nu$; see also \cite[\S~2]{K} and \cite[Chapter 8]{Ku3}. We refer to these texts for the following definitions and facts:\par

A \emph{pseudo--convergent} sequence of elements of a field $K$ endowed with a valuation $\nu$ is a family $(y_\tau)_{\tau\in T}$ of elements of $K$ indexed by a well ordered set $T$ without last element, which satisfies the condition that whenever $\tau<\tau'<\tau"$ we have $\nu(y_{\tau'}-y_\tau)<\nu(y_{\tau"}-y_{\tau'})$. \par

An element $y\in K$ is said to be a \emph{pseudo--limit}, or simply \emph{limit} of this pseudo--convergent sequence if $\nu(y_{\tau'}-y_\tau)\leq \nu(y-y_\tau)$ for $\tau ,\tau'\in T,\ \tau<\tau'$. One observes that if $(y_\tau)_{\tau\in T}$ is pseudo--convergent, then for each $\tau\in T$ the value $\nu(y_{\tau'}-y_\tau)$ is independent of $\tau'>\tau$ and can be denoted by $\nu_\tau$. Moreover, given $y\in K$, either $\nu(y-y_{\tau'})>\nu (y-y_\tau)$ whenever $\tau'>\tau$ (in which case $y$ is a limit), or there exists $\tau_0\in T$ such that $\nu(y-y_{\tau'})=\nu (y-y_\tau)$ for $\tau'>\tau>\tau_0$. In other words, the sequence $(\nu(y-y_{\tau}))_{\tau\in T}$ is either strictly increasing or \emph{eventually constant}. Taking $y=0$ we see that either $\nu(y_{\tau'})>\nu(y_\tau)$ whenever $\tau'>\tau$ or there exists $\tau_0\in T$ such that $\nu(y_{\tau'})=\nu (y_\tau)$ for $\tau'>\tau>\tau_0$. Finally, if $y$ and $z$ are two limits of $(y_\tau)_{\tau\in T}$, we have that for all $\tau\in T$, $\nu(y-z)> \nu_\tau$ since $T$ has no last element.\par 

In this paper, we mostly apply a variant of Ostrowski's method to pseudo--convergent sequences similar to the partial sums $(\sigma_i)_{i\geq1}$ for a purpose different from his. From now on we fix an algebraic closure $\widebar K$ of $K$ and recall that valuations of $K$ extend to $\overline K$ (see \cite[Ch.~VI, \S~1, no. 3, Theorem 3]{B}).\par\noindent
\begin{proposition}\label{propanddef}
Let $F(X)\in R[X]$ be a Nagata polynomial. Given an extension $\tilde\nu$ of $\nu$ to $\widebar K$, there exists a unique root of $F(X)$ in $\widebar K$ with positive $\tilde\nu$-value. If we call $\si$ this root of $F(X)$, then the following also holds:
\begin{enumerate}
\item $\si$ is a limit of the pseudo--convergent sequence $(\sigma_i)_{i\geq1}$ associated to $F(X)$.
\item\label{zero} For any $z\in\widebar K\setminus\set{\si}$ such that $F(z)=0$ we have $\tilde\nu(z)=0$.
\item $\si$ is a simple root of $F(X)$.
\end{enumerate}
\end{proposition}

\begin{proof}
Write $F(X)=\prod_{j=1}^n{X-r_j}$ in $\widebar K[X]$. For all $i\geq1$, we have $\nu(F(\sigma_i))=\sum_{j=1}^n{\tilde\nu(\sigma_i-r_j)}$. Hence if none of the $r_j$ is a limit of the pseudo--convergent sequence $(\sigma_i)_{i\geq1}$ then $(\nu(F(\sigma_i)))_{i\geq1}$ is eventually constant. However $\nu(F(\sigma_i))=\nu(\delta_i)$ for all $i\geq1$, so we can assume that $r_1$ is a limit of $(\sigma_i)_{i\geq1}$. In particular, $\tilde\nu(\sigma_i-r_1)=\nu(\sigma_{i+1}-\sigma_i)=\nu(\delta_i)$ for all $i\geq1$.

For $1\leq j\leq n$, we have $\tilde\nu(r_j)\geq0$ because $r_j$ is integral over $R$ (see \cite[Ch.~VI, \S~1, no.3, Theorem 3]{B}). In addition, $\nu(\sigma_i)=\nu(\delta_0)=\nu(F(0))=\sum_{j=1}^n{\tilde\nu(r_j)}$ for all $i\geq1$. If $\nu(\sigma_i)>\tilde\nu(r_1)$ for some $i$, we obtain $\nu(\delta_i)=\tilde\nu(\sigma_i-r_1)=\tilde\nu(r_1)<\nu(\delta_0)$, which gives us a contradiction. We conclude that $\tilde\nu(r_j)=0$ if $j\neq1$ and $\tilde\nu(r_1)=\nu(\delta_0)>0$.\end{proof}

In what follows, we keep the notation introduced in Proposition~\ref{propanddef} and an extension $\tilde\nu$ of $\nu$ to $\widebar K$ comes with a distinguished root $\si\in\widebar K$ of $F(X)$ satisfying $\tilde\nu(\si)=\nu(\delta_0)>0$. In the notation we will omit the dependence of $\si$ on $\tilde\nu$. Two extensions of $\nu$ may choose different roots of $F(X)$ in $\widebar K$, however their minimal polynomials over $K$ coincide in view of the following:\par\noindent   
\begin{corollary}\label{minpol}
In the situation of Proposition~\ref{propanddef}, denote by $F^*(X)$ the minimal polynomial of $\si$ over $K$. Then $F^*(X)$ is the only irreducible factor of $F(X)$ in $K[X]$ such that the $\nu$-value of its constant term is positive.
\end{corollary}

\begin{proof}
The $\nu$-value of the constant term of a unitary polynomial in $K[X]$ is the sum of the $\tilde\nu$-values of all its roots in $\widebar K$. Use that $\tilde\nu(\si)>0$ and Proposition~\ref{propanddef}.\eqref{zero} to prove the statement.\end{proof}

Throughout this section, we denote by $\widebar R$ the integral closure of $R$ in $K$. Recall that $\widebar R\subseteq R_\nu$. The localization $\tilde R=\widebar R_{m_\nu\cap\widebar R}$ is an integrally closed local domain dominating $R$ and dominated by $R_\nu$, which is uniquely determined by $\nu$. Note that a Nagata polynomial in $R[X]$ can also be regarded as a Nagata polynomial in $\tilde R[X]$.\par\noindent
\begin{lemma}\label{intclosure}
Keeping the same notation, we have the following:
\begin{enumerate}
\item\label{enum:defminpol} The coefficients of $F^*(X)$ belong to $\widebar R$.
\item\label{enum:stilde} The polynomials $F(X)$ and $F^*(X)$ are Nagata polynomials in $\tilde R[X]$ and they define the same Nagata extension of $\tilde R$.
\end{enumerate}
\end{lemma}

\begin{proof}
Since $\si$ is integral over $R$, its minimal polynomial over $K$ belongs to $\widebar R[X]$ (see \cite[Ch.~5, \S~1, no.3, Corollary]{B}). Next we prove the second statement.

By Corollary~\ref{minpol}, $F^*(0)\in m_\nu$. Hence the constant term of $F^*(X)\in\widebar R[X]\subseteq\tilde R[X]$ is in the maximal ideal of $\tilde R$. The result follows from Lemma~\ref{fac} applied to $\tilde R$ and $F(X)$. The natural epimorphism $\tilde R[X]/(F(X))\to\tilde R[X]/(F^*(X))$ induces an isomorphism of $R$--algebras from the Nagata extension of $\tilde R$ defined by $F(X)$ to that defined by $F^*(X)$.\end{proof}

After what we have just seen, the valuation $\nu$ determines an irreducible factor $F^*(X)\in\widebar R[X]$ of $F(X)$ in $K[X]$ with $\nu(F^*(0))=\nu(\delta_0)>0$. Denote by $K^*$ the field $K[X]/(F^*(X))$. Then the natural homomorphism $R[X]/(F(X))\to K^*$ induces a homomorphism of $R$-algebras 
\[E_S(\nu)\colon S=(R[X]/(F(X)))_{\mathcal N}\longrightarrow K^*.\]
Indeed, if $F^*(X)$ divides $P(X)\in R[X]$ in $\widebar R[X]$ (or, equivalently, in $K[X]$) then $P(0)=c\:F^*(0)$ for some $c\in\widebar R$ and, since $\nu$ is centered in $R$ and $F^*(0)\in m_\nu$, we have $P(0)\in m_R$.\par\noindent
\begin{definition}
Let $\nu$ be a valuation centered in the local domain $R$ and let $S$ be a Nagata extension of $R$ determined by the polynomial $F(X)\in R[X]$. The kernel of the homomorphism $E_S(\nu)\colon S\to K^*$ defined above is denoted by $H_S(\nu)$.
\end{definition}

\begin{remark}\label{firstsimpl} 
Let us go back to Lemma~\ref{intclosure}. If $R\to S$ and $\tilde R\to\tilde S$ are the Nagata extensions defined by $F(X)$, then we have a commutative diagram 
\begin{equation} \tag{$\text{D}_1$}
\begin{gathered}
\xymatrixcolsep{3pc}
\xymatrix@R-1pc{
R \ar[d] \ar[r] & S \ar[d] \ar[r]^-{E_S(\nu)}& K^*\\
\tilde R \ar[r] & \tilde S \ar[ru]_-{E_{\tilde S}(\nu)} }
\end{gathered}
\label{diagram1}
\end{equation}
where $S\to\tilde S$ is the local ring homomorphism induced by the natural map from $R[X]/(F(X))$ to $\tilde R[X]/(F(X))$. By Lemma~\ref{intclosure}.\eqref{enum:stilde}, $\tilde S$ is also defined by $F^*(X)$. As a consequence, $H_{\tilde S}(\nu)$ is the zero ideal and $H_S(\nu)$ is the kernel of $S\to\tilde S$. In the case where $R$ is integrally closed, we get that $E_S(\nu)$ is injective and $S$ is a local domain, and therefore $R^h$ is also a local domain.
\end{remark}

Observe that the ideal $H_S(\nu)$ depends only on the valuation $\nu$. It has the following properties:\par\noindent
\begin{lemma}\label{kernel}Let $\nu$ be a valuation centered in a local domain $R$. Then:
\begin{enumerate}
\item\label{enum:ker} For any Nagata extension $R\to S$, the ideal $H_S(\nu)$ of $S$ is a minimal prime.
\item\label{enum:compatible} Given a map $f\colon S\to S'$ of Nagata extensions of $R$, we have $f^{-1}(H_{S'}(\nu))=H_S(\nu)$.
\end{enumerate}
\end{lemma}

\begin{proof}Let $F(X)\in R[X]$ be a Nagata polynomial defining the extension $R\to S$. Let $\mathfrak p$ be a prime ideal of $S$ such that $\mathfrak p\cap R=(0)$. Then $\mathfrak p$ is the extension of a prime ideal of $R[X]/(F(X))$ that is contained in the maximal ideal $\mathcal N$ and has intersection $(0)$ with $R$. Since $R\to R[X]/(F(X))$ is an integral extension, we have by the incomparability property that $\mathfrak p$ is a minimal prime of $S$. To prove \eqref{enum:ker}, take $\mathfrak p=H_S(\nu)$.

As we saw, a morphism $f\colon S\to S'$ of Nagata extensions is determined by an element $\xi'\in m_{S'}$ such that $F(\xi')=0$. Then, with the notations of Remark~\ref{firstsimpl}, the image $\tilde{\xi'}$ of $\xi'$ under the local ring homomorphism $S'\to\tilde{S'}$ determines a map $\tilde f:\tilde S\to\tilde{S'}$ of Nagata extensions of $\tilde R$. We see that mapping the image of $X$ in $K^*$ to the image of $\tilde{\xi'}$ in the field $K'^*$ uniquely defines a map $K^*\to K'^*$ of extensions of $K$, which has to be injective and and makes the following diagram commute:
\begin{equation*} 
\begin{gathered}
\xymatrixcolsep{3pc}
\xymatrix@R-1pc{
S \ar[d]_-{f}\ar[r] & \tilde S \ar[d]_-{\tilde f}\ar@{^{(}->}[r]^-{E_{\tilde S}(\nu)} & K^*\ar[d] \\
S' \ar[r] & \tilde{S'} \ar@{^{(}->}[r]^-{E_{\tilde{S'}}(\nu)} & K'^* }
\end{gathered}
\end{equation*}
Statement $(2)$ now follows.\end{proof}

\begin{Remark}\label{rem:minprimes} 
\begin{enumerate}[wide,topsep=0mm,partopsep=0mm,parsep=0mm,itemsep=0.1mm,labelindent=0pt,leftmargin=0pt]
\item By a direct computation one can check that $H_S(\nu)$ is the extension with respect to the canonical localization homomorphism $R[X]/(F(X))\to S$ of the prime ideal consisting of the residue classes of the polynomials in $R[X]$ that are divisible by $F^*(X)$ in $K[X]$.
\item\label{enum:flat} As we have seen in the proof of Lemma~\ref{kernel}.\eqref{enum:ker}, if $R\to S$ is a Nagata extension, then any prime ideal of $S$ lying over the zero ideal of $R$ is a minimal prime. Conversely, if $\mathfrak p$ is a minimal prime ideal of $S$ then $\mathfrak p\cap R=(0)$. This follows from the fact that $R\to S$ is flat by using that going--down property holds for flat extensions, see \cite[(5.D) Theorem 4]{Mat}. 
\end{enumerate}
\end{Remark}

Let us prove that the valuation $\nu$ uniquely determines the support of the semivaluation which extends it to the henselization:

\begin{proposition}\label{uniqminprime}
Let $\nu$ be a valuation centered in a local domain $R$ and let $R\to S$ be a Nagata extension. If $\mathfrak p$ is a prime ideal of $S$ such that $\mathfrak p\cap R=(0)$ and $\nu$ extends to a valuation centered in $S/\mathfrak p$ through the inclusion $R\subset S/\mathfrak p$, then $\mathfrak p=H_S(\nu)$.
\end{proposition}

\begin{proof}
Taking into account what we saw in the proof of Lemma~\ref{kernel}.\eqref{enum:ker}, the ideal $\mathfrak p$ corresponds in $R[X]$ to a minimal prime ideal $\mathfrak q$ over $(F(X))$ such that $\mathfrak q\cap R=(0)$ and $\mathfrak q\subseteq (m_R,X)$. Using that $K[X]$ is a principal ideal domain and that $R\to R[X]/(F(X))$ is an integral extension, we see that $\mathfrak q$ consists of the polynomials in $R[X]$ that are divisible in $K[X]$ by some irreducible factor $Q(X)\in K[X]$ of $F(X)$. We can write $Q(X)=X^s+q_1X^{s-1}+\ldots+q_0$ with $q_j\in\widebar R$ for all $j$ (see \cite[Ch.~5, \S~1, no.3, Corollary]{B}). Denote by $\bar x$ the image of $X$ in $S/\mathfrak p$. We have a valuation $\mu$ centered in $S/\mathfrak p$ which  extends $\nu$. It satisfies $\mu(\bar x)>0$ and $\mu(q_j)\geq0$ for all $j$. The relation $\bar x^s+q_1\bar x^{s-1}+\ldots+q_0=0$ implies that $\nu(Q(0))=\mu(q_0)>0$. Since the field extension $K\to\text{Frac}(S/\mathfrak p)$ is algebraic, we can embed $\text{Frac}(S/\mathfrak p)$ in $\widebar K$ and extend $\mu$ to $\widebar K$. By Corollary~\ref{minpol} we have $Q(X)=F^*(X)$, and therefore $\mathfrak p=H_S(\nu)$.\end{proof}

Given an extension $\tilde\nu$ of $\nu$ to $\widebar K$, the evaluation at $\si$, namely $P(X)\mapsto P(\si)$, induces a $K$--isomorphism of fields $\pi_{\si}\colon K^*\to K(\si)\subset\widebar K$ (recall that $F^*(X)$ is the minimal polynomial of $\si$ over $K$). The image of the composition 
\[\pi_{\si}\circ E_s(\nu):S\to K^*\to K(\si)\]
is the local $R$--subalgebra $R[\si]_{(m_R,\si)}$ of $K(\si)$. Since the ideal is clear from the context, let us denote it simply by $R[\si]_*$. Observe that the quotient $S/H_s(\nu)$ is naturally isomorphic to $R[\si]_*$ and the restriction of $\tilde\nu$ to $K(\si)$ is centered in $R[\si]_*$ because $\tilde\nu(\si)>0$ and $\nu$ is centered in $R$. In this way $\tilde\nu$ determines a valuation $\tilde\nu_S$ centered in $S/H_S(\nu)$ which is an extension of $\nu$. Next we prove the uniqueness of this extension.\par\noindent
\begin{proposition}\label{uniqextension}
Keep the notation of Proposition~\ref{uniqminprime}. There is a unique valuation centered in $S/H_S(\nu)$ which extends $\nu$ through the inclusion $R\subset S/H_S(\nu)$.
\end{proposition}

\begin{proof}
Any such extension of $\nu$ can be obtained in the way explained above starting from an extension to $\widebar K$. Therefore it suffices to take two extensions $\tilde\nu$ and $\tilde\nu'$ of $\nu$ to $\widebar K$ and show that $\tilde\nu_S=\tilde\nu'_S$. In that situation, by \cite[Ch.~VI, \S~7, Corollary 3]{ZS}, there exists a $K$--automorphism $\pi$ of $\widebar K$ such that $\tilde\nu'=\tilde\nu\circ\pi$. Let $\si$ and $\si'$ be the distinguished roots of $F(X)$ in $\widebar K$ associated to $\tilde\nu$ and $\tilde\nu'$, respectively (see Proposition~\ref{propanddef}). Since $\tilde\nu'(\pi^{-1}(\si))=\tilde\nu(\si)>0$, the automorphism $\pi$ must send $\si'$ to $\si$. Recall that $\si$ and $\si'$ have the same minimal polynomial over $K$ by Corollary~\ref{minpol}. We have $\pi_{\si}=\left.\pi\right|_{K(\si')}\circ\pi_{\si'}$ and $\tilde\nu_S=\tilde\nu'_S$.\end{proof}

We are now in position to prove Theorem~\ref{mainthm}.\eqref{hext}.\par\noindent
\begin{proof}[Proof of Theorem~\ref{mainthm}.\eqref{hext}]
The henselization $R^h$ is the inductive limit of the Nagata extensions $S$ of $R$. For every $S$, denote by $f_S:S\to R^h$ the canonical local ring homomorphism. By Lemma~\ref{kernel}.\eqref{enum:compatible}, the inductive limit $H(\nu)=\varinjlim H_S(\nu)$ is well defined and it is an ideal of $R^h$ such that $f_S^{-1}(H(\nu))=H_S(\nu)$ for all $S$. Moreover, it is a minimal prime and lies over the zero ideal of $R$, because all the $H_S(\nu)$ satisfy this and $\mathfrak p=\varinjlim f_S^{-1}(\mathfrak p)$ for any ideal $\mathfrak p$ of $R^h$. The domain $R^h/H(\nu)$ is the union of the $S/H_S(\nu)$. Since the valuation $\nu$ determines uniquely each $H_S(\nu)$ (recall its definition and Proposition~\ref{uniqminprime}), and according to Proposition~\ref{uniqextension} it extends uniquely to each $S/H_S(\nu)$, the same is true for $R^h/H(\nu)$.\end{proof}


\section{Effective computation of the extended valuation}\label{sec3} 

Let us keep the notation introduced in Section~\ref{sec2}. In particular, we have fixed a valuation $\nu$ centered in a local domain $R$ and a non trivial Nagata extension $R\to S$ defined by a Nagata polynomial $F(X)\in R[X]$. We present the quotient $S/H_S(\nu)$ in the form $R[\si]_*$ and call $\tilde\nu$ the unique valuation centered in $R[\si]_*$ extending $\nu$ (see Proposition~\ref{uniqextension} and the paragraph before it).\par 

In this section we show how we can compute the $\tilde\nu$-value of an arbitrary non zero element of $R[\si]_*$. Observe that it is enough to study the values of the form $\tilde\nu(h(\si))$, where $h(X)$ is a polynomial in $R[X]$ such that $0\leq\deg h(X)<\deg F^*(X)$. As a direct consequence, we will obtain that the value group of $\nu$ coincides with the value group of its extension $\tilde\nu$. From this, we deduce that Theorem~\ref{mainthm}.\eqref{grouphext} holds.

\subsection{Initial forms with respect to a valuation}\label{review}
We begin by recalling some definitions and results that we need for the understanding of the rest of the section.\par

A valuation $\nu$ centered in a local domain $R$ determines a filtration on $R$ indexed by the semigroup of values $\Gamma=\nu(R\setminus\set{0})$. This filtration is defined by the ideals 
\[{\mathcal P}_\varphi (R)=\{z\in R\,\vert\, \nu (z)\geq \varphi\}\text{ and }{\mathcal P}^+_\varphi (R)=\{z\in R\,\vert\, \nu (z)> \varphi\}.\]
To this filtration is associated a graded ring
\[{\rm gr}_{\nu}R=\bigoplus_{\varphi\in\Gamma}\frac{{\mathcal P}_\varphi (R)}{{\mathcal P}^+_\varphi (R)},\]
where each element $z$ of $R\setminus\left\lbrace 0\right\rbrace$ has a non zero \emph{initial form} ${\rm in}_{\nu}z$, its image in the quotient $\frac{{\mathcal P}_{\nu (z)} (R)}{{\mathcal P}^+_{\nu (z)} (R)}$. By construction, the value of the difference of two elements is larger than the value of each if and only if they have the same initial form.\par

Let $K$ be the fraction field of $R$ and let $\Phi$ be the value group of the valuation $\nu$, that is, $\Phi=\nu(K\setminus\set{0})$. The \emph{rank}, or \emph{height}, of $\nu$ is the the \emph{cardinal}\footnote{The set of convex subgroups of $\Phi$ may not be well ordered; see \cite[Exercise 3 to Chap.VI, \S~4]{B}, where it appears that the set of those convex subgroups of a totally ordered abelian group that are \emph{principal} can realize any totally ordered set. However, the smallest convex subgroup containing a subset of $\Phi$ exists as the intersection of such convex subgroups.} of the totally ordered set (for the order opposite to inclusion) of nonzero prime ideals of $R_\nu$, or equivalently of the totally ordered set (for inclusion) of convex subgroups of $\Phi$ different from $\Phi$. We refer to \cite[Chapter VI, Theorem 15]{ZS} or \cite[Chapter VI, \S~4, no. 5, Definition 2]{B} for details. If the \emph{rational rank} ${\rm dim}_\Q\Phi\otimes_\Z\Q$ of $\nu$ is finite, for example if the ring $R$ is noetherian, (see \cite[Appendix 2, Proposition 1, Proposition 2]{ZS}), the rank is finite.\par

Let $\Psi$ be a proper convex subgroup of $\Phi$, $\Psi\neq(0)$. Let $m_\Psi$ (resp. $p_\Psi$) be the prime ideal of $R_\nu$ (resp. $R$) corresponding to $\Psi$, that is,
\[m_\Psi=\set{x\in R_\nu\,\vert\,\nu(x)\notin\Psi}\text{ and }p_\Psi=m_\Psi\cap R.\]
The valuation $\nu$ is composed of a \emph{residual valuation} $\bar\nu_\Psi$, whose valuation ring $R_{\bar\nu_\Psi}$ is the quotient $R_\nu/m_\Psi$ and with values in $\Psi$, and a valuation $\nu'_\Psi$ whose valuation ring is the localization $R_{m_\Psi}$ and with values in $\Phi/\Psi$. With the usual notation, $\nu=\nu'_\Psi\circ\bar\nu_\Psi$. For every $x\in R_\nu\setminus m_\Psi$, we have  $\bar\nu_\Psi(\bar x)=\nu(x)$, where $\bar x$ denotes the residue class of $x$ in $R_\nu/m_\Psi$. We have an injective local ring map $R/p_\Psi\hookrightarrow R_\nu/m_\Psi$ and the valuation $\bar\nu_\Psi$ induces by restriction a valuation centered in $R/p_\Psi$ (with value group contained in $\Psi$). We denote this valuation also by $\bar\nu_\Psi$ and call it the \emph{residual valuation on $R/p_\Psi$}. We extend its definition to the case of $\Psi=\Phi$ setting $p_\Phi=(0)$ and $\bar\nu_\Phi=\nu$.\par

Let $K\hookrightarrow L$ be an algebraic field extension. For any valuation $\tilde\nu$ of $L$ extending $\nu$, by \cite[Ch.~VI, \S~8, no. 1, Proposition 1]{B} the value group $\tilde\Phi$ of $\tilde\nu$ contains $\Phi$ as a subgroup such that $\tilde\Phi/\Phi$ is a torsion group . By \cite[Ch.~VI, \S~8, no. 1, Lemme 2]{B}, if $[L:K]$ is finite $\Phi$ is a subgroup of finite index of $\tilde\Phi$. The fact that $\tilde\Phi/\Phi$ is a torsion  group implies that the map $\tilde\Psi\mapsto \tilde\Psi\cap\Phi$ is an ordered bijection from the set of convex subgroups of $\tilde\Phi$ to the set of convex subgroups of $\Phi$. The inverse map associates to a convex subgroup $\Psi\subset\Phi$ the smallest convex subgroup of $\tilde\Phi$ containing $\Psi$ (see  \cite[Ch.~VI, \S~8, no. 1, Lemme 3]{B}). Moreover, each quotient $\tilde\Psi/\Psi$ is a torsion group so that $\Psi$ is cofinal in $\tilde\Psi$.

\subsection{Ostrowski's Lemma and initial forms}\label{subsec:Taylorexp}
Now return to our setting. Let us consider the sequences $(\delta_i)_{i\in\N}$ and $(\sigma_i)_{i\geq1}$ attached to the Nagata polynomial $F(X)\in R[X]$ (see Definition~\ref{def:sequences}). For $i\geq1$, set 
\[\eta_i:=\si -\sigma_i\in R[\si]_*.\]
By Proposition~\ref{propanddef}, we have that $\tilde\nu(\eta_i)=\nu(\delta_i)$ for all $i\geq1$.\par 

Let $h(X)$ be a polynomial in $R[X]$ of degree $s\geq 0$. We note that $h(\sigma_i)\in R$ for all $i\geq1$ and we are going to study the behavior of the $\nu(h(\sigma_i))$ as $i$ increases. Since the Nagata extension is non trivial, the $\sigma_i$ are all different and thus $h(\sigma_i)\neq 0$ for all $i$ large enough.\par

Consider the usual expansion $h(X+\alpha)=\sum_{m=0}^s h_m(X)\alpha^m$ of $h(X+\alpha)$ as a polynomial in $X$ and $\alpha$. If the polynomial $h_m(X)$ is not zero, then its degree is $s-m$.\par\noindent
\begin{remark}\label{HasseSchmidt}
The maps $\partial_m\colon h(X)\mapsto h_m(X)$ are Hasse--Schmidt derivations satisfying the identities $\partial_{m}\circ \partial_{m'}=\partial_{m'}\circ \partial_{m}=\binom{m+m'}{ m}\partial_{m+m'}.$ Some use the mnemonic notation $\partial_m=\frac{1}{m!\ }\frac{\partial^m}{\partial X^m}$.
\end{remark}

We have the following identities in $K(\si)$:
\begin{align}
h(\sigma_\infty) = h(\sigma_i)+\sum_{m=1}^sh_m(\sigma_i)\eta_i^m,\label{eq:1}\tag{$\star$}\\
h(\sigma_i)=h(\sigma_\infty)+\sum_{m=1}^sh_m(\sigma_\infty)(-1)^m\eta_i^m.\label{eq:3}\tag{$\star\star$}
\end{align}
Since $\sigma_{i+1}=\sigma_i+\delta_i$, we also have the identity:
\begin{equation}\label{eq:2}\tag{$\star\star\star$}
h(\sigma_{i+1})=h(\sigma_i)+\sum_{m=1}^sh_m(\sigma_i)\delta_i^m.
\end{equation}

\begin{lemma}\label{rat} 
The subgroup of the value group $\Phi$ of $\nu$ generated by the values of the $\delta_i$ is finitely generated and therefore of finite rational rank. 
\end{lemma} 
\begin{proof} 
Starting from the equalities $F'(\sigma_i)\delta_i=-F(\sigma_i)=-\sum_{m=1}^nF_m(\sigma_\infty)(-1)^m\eta_i^m$ which follow from \eqref{eq:3} and applying again the equality \eqref{eq:3} to the derivative $F'(X)=F_1(X)$, we obtain the following equality, where the $F'_q(X)$ are those polynomials occurring in the expansion $F'(X+\alpha)=F'(X)+\sum_{q=1}^{n-1}F'_q(X)\alpha^q$:
\[F'(\sigma_\infty)\eta_{i+1}=\sum_{m=2}^nF_m(\sigma_\infty)(-1)^m\eta_i^m+\sum_{q=1}^{n-1}F'_q(\sigma_\infty)(-1)^q\eta_i^q\delta_i.\]
By construction, we have the identity $F'_q(X)=(q+1)F_{q+1}(X)$ for $1\leq q\leq n-1$, so that the previous equality can be rewritten as
\[F'(\sigma_\infty)\eta_{i+1}=\sum_{m=2}^n(\eta_i-m\delta_i)F_m(\sigma_\infty)(-1)^m\eta_i^{m-1}.\]
For $2\leq m\leq n$, we have $\tilde\nu(\eta_i-m\delta_i)=\tilde\nu(\eta_{i+1}-(m-1)\delta_i)=\nu(\delta_i)$. Since the sequence $(\nu(\delta_i))_{i\in\N}$ is strictly increasing, the $\tilde\nu$-values of any two terms of the sum of the right hand side are different for all large enough $i$. Therefore, remembering that $\tilde\nu(F'(\sigma_\infty))=0$ and that none of the $\delta_i$ or $\eta_i$ is zero (because the extension $R\to S$ is non trivial), this equality shows that $\tilde\nu(\eta_{i+1})=\nu(\delta_{i+1})$ is, for large $i$, in the semigroup generated by the $\nu$-values of finitely many $\delta_i$ and the $\tilde\nu$-value of the $F_m(\sigma_\infty)$, $m=2,\ldots,n$. Thus, the subgroup of $\Phi$ generated by the $\nu(\delta_i)$ is contained in a finitely generated subgroup of the value group $\tilde\Phi$ of $\tilde\nu$ and so is finitely generated.\end{proof}

\begin{remark} The fact that the group generated by the $\nu(\delta_i)$ is of finite rational rank also follows from Abhyankar's inequality since this group is contained in the value group of the restriction of $\nu$ to a subring of $R$ which is essentially of finite type over the prime ring and contains the coefficients of $F(X)$. Lemma~\ref{rat} gives a stronger result. \end{remark}

Since the sequence $(\nu(\delta_i))_{i\in\N}$ is strictly increasing, for large $i$ there are no two terms of the sum in \eqref{eq:3} with the same $\tilde\nu$-value. However we do not see immediately that the permutation $(k_1,\ldots,k_s)$ of 
$(1,\ldots,s)$ such that 
$\tilde\nu(h_{k_1}(\sigma_\infty)(-1)^{k_1}\eta_i^{k_1})<\ldots<
\tilde\nu(h_{k_s}(\sigma_\infty)(-1)^{k_s}\eta_i^{k_s})$ is independent of $i$. 
The purpose of the next proposition is to establish this.\par\noindent
\begin{proposition}{\rm (Ostrowski-Kaplansky; see \cite[Teil III, statement IV, p.371 ff.]{O} and \cite[Lemma 4]{K})} \label{OK} Let $\Phi$ be a totally ordered abelian group. Let $\beta_1,\ldots ,\beta_s\in \Phi$ and distinct integers $t_1,\ldots ,t_s\in \N\setminus\{0\}$ be given. Let $(\gamma_\tau)_{\tau\in T}$ be a strictly increasing family of elements of $\Phi$ indexed by a well ordered set $T$ without last element. There exist  an element $\iota\in T$ and a permutation $(k_1,\ldots ,k_s)$ of $(1,\ldots ,s)$ such that for all $\tau\geq\iota$ we have the inequalities
$$\beta_{k_1}+t_{k_1}\gamma_\tau< \beta_{k_2}+t_{k_2}\gamma_\tau<\cdots <\beta_{k_s}+t_{k_s}\gamma_\tau.$$
\end{proposition}
\begin{proof}\footnote{We are grateful to Laurent Moret-Bailly for communicating to us this elegant proof, which in particular makes the original assumption that the set $T$ is well ordered superfluous.}
 We begin with the following:
\begin{proposition}{\rm (L. Moret-Bailly)}\label{laurent}Let $\Phi$ be a totally ordered abelian group and $T$ a totally ordered non empty set without last element. Let $\gamma\colon T\to \Phi;\ \tau\mapsto \gamma_\tau$, be a strictly increasing map. Consider the group $H=\Phi\times \Z$ and define $H_+$ to be the set of pairs $(\phi,m)\in H$ such that $\phi +m\gamma_\tau \geq 0$ for large enough $\tau$. Then $H_+$ is the positive semigroup of a totally ordered group structure on $H$.
\end{proposition}
\begin{proof} It follows almost directly from the assumptions that $H$ is the direct sum of $H_+$ and $-H_+=\{(\phi,m)\vert \phi+m\gamma_\tau\leq 0\ \hbox{\rm for large enough }\ \tau\}$. Indeed, if $(\phi,m)\notin H_+$, then if $m>0$ we must have $\phi+m\gamma_\tau<0$ for large $\tau$ since the map $\gamma$ is strictly increasing and if $m<0$ we use the fact that $\tau\mapsto \phi+m\gamma_\tau$ is strictly decreasing. This gives $H$ the desired structure.
\end{proof}
The result of Ostrowski-Kaplansky follows : the pairs $(\beta_i,t_i)$, which are distinct since the $t_i$ are, lie in the totally ordered group $H$ and therefore, up to a permutation, are strictly ordered with respect to the order on $H$ determined by the map $\gamma$.
\end{proof}

\begin{remark}\label{Ku}Another proof of Proposition~\ref{OK} is given by F.-V. Kuhlmann in \cite[Lemma 8.8]{Ku3}, which is unfortunately not yet published.\end{remark}

We can now use Proposition~\ref{OK} to prove the\par\noindent
\begin{proposition}\label{asym} Let $h(X)\in R[X]$ be a polynomial of degree $s>0$. Keep the notations of the identity \eqref{eq:3} above. There exist $i_0\in\N$ and $k\in\{1,\ldots,s\}$ such that for $i\geq i_0$ we have
\[{\rm in}_{\tilde\nu}(h(\sigma_\infty)-h(\sigma_i))=-{\rm in}_{\tilde\nu}(h_k(\sigma_\infty)(-1)^k\eta_i^k).\] In particular, $\tilde\nu(h(\sigma_\infty)-h(\sigma_i))=\tilde\nu(h_k(\sigma_\infty)(-1)^k\eta_i^k)$ for $i\geq i_0$ and $h(\sigma_\infty)$ is a limit for the valuation $\tilde\nu$ of the pseudo--convergent sequence $(h(\sigma_i))_{i\geq i_0}$.
\end{proposition}
\begin{proof} 
The polynomial $h_s(X)$ is a nonzero constant polynomial. If $s=1$ then the first statement is trivial. In the general case it is enough to apply Proposition~\ref{OK} to the $\beta_m=\tilde\nu(h_m(\sigma_\infty))$ in the value group $\tilde\Phi$ of $\tilde\nu$, with  $t_m=m$, $\gamma_i= \nu(\delta_i)$, and $T=\N$, recalling that $\tilde\nu(\eta_i)=\nu(\delta_i)$ and $\nu(\delta_{i+1})\geq 2\nu(\delta_i)$.

The second statement follows directly from the first. In order to prove the last part of the proposition it is enough to observe that if $h(\sigma_\infty)\in K(\sigma_\infty)$ is such that $\tilde\nu(h(\sigma_\infty)-h(\sigma_i))$ is less than $\tilde\nu(h(\sigma_\infty)-h(\sigma_j))$ for $i_0\leq i<j$, then $(h(\sigma_i))_{i\geq i_0}$ is necessarily a pseudo--convergent sequence for $\tilde\nu$ having $h(\sigma_\infty)$ as a limit. Since $\tilde\nu(\eta_i)<\tilde\nu(\eta_j)$ for $i<j$, this ends the proof.\end{proof}

\begin{Remark}
\begin{enumerate}[wide,topsep=0mm,partopsep=0mm,parsep=0mm,itemsep=0.1mm,labelindent=0pt,leftmargin=0pt]
\item Proposition~\ref{asym} is essentially a slightly more precise version of a result of Ostrowski (see \cite[part III, statement III, p.371 ff.]{O}, \cite[Lemma 5]{K} and \cite[Chapter 8]{Ku3}) to the effect that the values taken by a polynomial with coefficients in $R$ on a pseudo--convergent sequence of elements of $R$ (in this case the $\sigma_i$) form themselves a pseudo--convergent sequence and therefore their valuations are eventually either constant or strictly increasing. Ostrowski's result, proved for rank one valuations in \cite{O}, stated for arbitrary rank in \cite{K} and proved in \cite{Ku3}, is more general in that it applies to all pseudo--convergent sequences.
\item We shall see below in Proposition~\ref{ae} that there exist an index $i_0$, an element $a\in R_\nu$ and an integer $e$ with $0\leq e\leq s$ such that for $i\geq i_0$ we have ${\rm in}_\nu h(\sigma_i)={\rm in}_\nu (a\delta_i^e)$.
\end{enumerate}
\end{Remark}

\begin{corollary}\label{approx}Keep the notations of Proposition~\ref{asym}. If the sequence $(\nu(h(\sigma_i)))_{i\geq1}$ is eventually constant, then ${\rm in}_{\tilde\nu}(h(\sigma_\infty))={\rm in}_{\nu}(h(\sigma_i))$ for all $i$ large enough. Otherwise, for all $i$ large enough we have $\tilde\nu(h(\sigma_\infty))>\nu(h(\sigma_{i+1}))>\nu(h(\sigma_i))$ and ${\rm in}_{\nu}(h(\sigma_i))=-{\rm in}_{\tilde\nu}(h_k(\sigma_\infty)(-1)^k\eta_i^k)$.
\end{corollary}
\begin{proof} 
Since $(h(\sigma_i))_{i\geq i_0}$ is pseudo--convergent, then either there exists $i_1\geq i_0$ such that the sequence $(\nu(h(\sigma_i)))_{i\geq i_1}$ is constant or $\nu(h(\sigma_j))>\nu(h(\sigma_i))$ whenever $j>i\geq i_0$.

Assume that we are in the first case and call $\varphi=\nu(h(\sigma_i))$, $i\geq i_1$. Since $(\tilde\nu(\eta_i))_{i\geq1}$ is strictly increasing, for all $i$ large enough, $\varphi$ is different from $\tilde\nu(h_k(\sigma_\infty)(-1)^k\eta_i^k)$. By the same argument, $\tilde\nu(h(\sigma_\infty))\neq\tilde\nu(h_k(\sigma_\infty)(-1)^k\eta_i^k)$ for any sufficiently large $i$. Then, for large $i$ we must have $\tilde\nu(h(\sigma_\infty))=\nu(h(\sigma_i))<\tilde\nu(h(\sigma_\infty)-h(\sigma_i))=\tilde\nu(h_k(\sigma_\infty)(-1)^k\eta_i^k)$ and, as a consequence, ${\rm in}_{\tilde\nu}(h(\sigma_\infty))={\rm in}_{\nu}(h(\sigma_i))$.

If $\nu(h(\sigma_j))>\nu(h(\sigma_i))$ for $j>i\geq i_0$, then $y=0$ is a limit for $\tilde\nu$ of the pseudo--convergent sequence $(h(\sigma_i))_{i\geq i_0}$. If $h(\si)=0$ then the result is clear because $h(\sigma_i)\neq0$ for large $i$, so assume that $h(\si)\neq0$. From the fact that $h(\sigma_\infty)$ is also a limit we deduce that \[\tilde\nu(h(\sigma_\infty))=\tilde\nu(h(\sigma_\infty)-y)>\nu(h(\sigma_j)-h(\sigma_i))=\nu(h(\sigma_i))\] for $j>i\geq i_0$. Hence for $i\geq i_0$, ${\rm in}_{\tilde\nu}(h(\sigma_\infty)-h(\sigma_i))=-{\rm in}_{\nu}(h(\sigma_i))$, which coincides with ${\rm in}_{\tilde\nu}(h_k(\sigma_\infty)(-1)^k\eta_i^k)$.\end{proof}

\subsection{The convex subgroups associated to a Nagata polynomial}\label{subsec:firststeps} 
In this subsection we denote by $\Phi$ the value group of $\nu$. We make the first steps towards the computation of $\tilde\nu$. Our main tools are  Lemma~\ref{quot}, Corollary~\ref{approx}, and two convex subgroups of $\Phi$ that we associate to a Nagata polynomial, one of them invariant after any change of variable $X\mapsto X'+\alpha$ in $R[X]$ with $\alpha\in m_R$.\par\noindent
\begin{definition}\label{def:noni}
Let $\nu$ be a valuation centered in a local domain $R$ and let $\Phi$ be its value group. 
Let $F(X)\in R[X]$ be a Nagata polynomial such that $\delta_i\neq0$ for all $i\geq0$, where $(\delta_i)_{i\in\N}$ is the Newton sequence of values attached to it. The \emph{convex subgroup $\Psi_F$ of $\Phi$ associated to $F(X)$} is the smallest convex subgroup of $\Phi$ containing all the $\nu(\delta_i)$.
\end{definition}

In particular the subgroup $\Psi_F$ is defined for Nagata polynomials giving rise to a non trivial Nagata extension. We have $\Psi_F\neq(0)$ because all the $\delta_i$ belong to $m_R$. We now come back to the behavior of the $\delta_i$ with the following two observations:\par\noindent
\begin{lemma}\label{cofinal}
Let $\tilde\Phi$ be the value group of $\tilde\nu$ and let $\widetilde{\Psi_F}$ be the smallest convex subgroup of $\tilde\Phi$ containing $\Psi_F$. The $\nu(\delta_i)$ are cofinal in $\Psi_F$ and therefore in $\widetilde{\Psi_F}$.
\end{lemma}
\begin{proof}
We have seen in Subsection~\ref{review} that $\Psi_F$ is cofinal in $\widetilde{\Psi_F}$. Next we prove that $(\nu(\delta_i))_{i\in\N}$ is cofinal in $\Psi_F$.

Observe that $\Psi_F$ is the smallest convex subgroup of $\Phi$ which contains the subgroup $\Phi'\subset\Phi$ generated by the $\nu(\delta_i)$. Let us show that $\Phi'$ is cofinal in $\Psi_F$. It suffices to deal with the positive semigroups of the groups in sight. If there exists $\varphi\in\Psi_{F,>0}$ such that for any $\zeta\in\Phi'_{>0}$ we have $\zeta<\varphi$, we see that the elements $\{\theta\in\Phi_{\geq 0}\,\vert\, \forall n\in\N,\ n\theta<\varphi\}$ form a semigroup: given $\theta, \theta'$ with this property we may assume that $\theta\leq \theta'$  and then $n(\theta+\theta')\leq 2n\theta'<\varphi$ for all $n\in\N$. This semigroup contains $\Phi'_{\geq 0}$, is the positive part of a convex subgroup of $\Phi$ and does not contain $\varphi$. This contradicts the minimality of $\Psi_F$. Thus, $\Phi'$ is cofinal in $\Psi_F$ and to finish the proof it suffices to show that the $\nu(\delta_i)$ are cofinal in $\Phi'$, which has finite rank by Lemma~\ref{rat}.

Let $\Psi'_1$ be the largest proper convex subgroup of $\Phi'$ (it exists because $\Phi'$ has finite rank). By construction it cannot contain all the $\nu(\delta_i)$. We know from Proposition~\ref{app} that $\nu(\delta_{i+1})\geq 2\nu(\delta_i)$. The images of the $\nu(\delta_i)$ in the rank one group $\Phi'/\Psi'_1$ satisfy the same inequality and therefore are cofinal in this archimedian ordered group. This implies that the $\nu(\delta_i)$ are cofinal in $\Phi'$.\end{proof}

\begin{lemma}\label{simple}
There exists $i_0\geq0$ such that $\Psi_F$ is the smallest convex subgroup of $\Phi$ containing $\nu(\delta_{i_0})$.
\end{lemma}
\begin{proof}
By Lemma~\ref{rat}, the subgroup $\Phi'$ of $\Phi$ generated by the $\nu(\delta_i)$ is finitely generated, so there exists $i_0\geq0$ such that $\Phi'=\Z\nu(\delta_{0})+\ldots+\Z\nu(\delta_{i_0})$. Therefore $\Psi_F$ is the smallest convex subgroup of $\Phi$ that contains the set $\set{\nu(\delta_i)}_{i=0}^{i_0}$. To finish the proof, use that $\nu(\delta_i)<\nu(\delta_{i+1})$.
\end{proof}

\begin{remark}
With the notations of Lemma~\ref{cofinal}, if $\Psi_F=\Phi$ then $\widetilde{\Psi_F}=\tilde\Phi$.
\end{remark}

Let $\bar\nu_{\Psi_F}$ be the residual valuation on $R/p_{\Psi_F}$ (see Subsection~\ref{review}). 
We call $L_F$ the fraction field of $R/p_{\Psi_F}$ and we fix an algebraic closure $\widebar L_F$ of $L_F$. Given $a\in R$, we denote by $\bar a$ the residue class of $a$ modulo $p_{\Psi_F}$.\par

The image $\widebar F(X)\in R/p_{\Psi_F}[X]$ of the Nagata polynomial $F(X)\in R[X]$ is again a Nagata polynomial. In addition, the Newton sequence of values and the sequence of partial sums attached to $\widebar F(X)$ are $(\bar\delta_i)_{i\in\N}$ and $(\bar\sigma_i)_{i\geq1}$, respectively (that is, they are obtained from the sequences attached to $F(X)$ by reduction modulo $p_{\Psi_F}$). By construction, $\bar\delta_i$ is different from zero and $\nu(\delta_i)=\bar\nu_{\Psi_F}(\bar\delta_i)$ for all $i\geq0$, and the convex subgroup $\Psi_{\widebar F}$ associated to $\widebar F(X)$ is the whole group of the valuation $\bar\nu_{\Psi_F}$.\par 

Hence we have a Nagata extension $R/p_{\Psi_F}\to S_F=(R/p_{\Psi_F}[X]/(\widebar F(X)))_{\mathcal{N}_F}$, where $\mathcal{N}_F$ corresponds to the maximal ideal $(m_R/p_{\Psi_F},X)$ of $R/p_{\Psi_F}[X]$; and a valuation $\bar\nu_{\Psi_F}$, in the sequel called $\bar\nu_F$, which is centered in the local domain $R/p_{\Psi_F}$. According to what we saw in Section~\ref{sec2}, the valuation $\bar\nu_F$ determines a minimal prime ideal $H_{S_F}(\bar\nu_F)$ of $S_F$ with the property that $\bar\nu_F$ extends uniquely to a valuation $\tilde{\bar\nu}_F$ centered in $S_F/H_{S_F}(\bar\nu_F)$ through the inclusion $R/p_{\Psi_F}\subset S_F/H_{S_F}(\bar\nu_F)$. We present the quotient $S_F/H_{S_F}(\bar\nu_F)$ in the form $R/p_{\Psi_F}[\bar\sigma_\infty]_*$ where $\bar\sigma_\infty\in\widebar L_F$ satisfies $\widebar F(\bar\sigma_\infty)=0$ and $\tilde{\bar\nu}_F(\bar\sigma_\infty)>0$.\par

We shall need the following lemma, in which we keep these notations. Observe that in the case where $\Psi_F=\Phi$ it also holds (one has $R/p_{\Psi_F}=R$, $S_F=S$, and $\bar\nu_F=\nu$).\par\noindent
\begin{lemma}\label{quot}
Let $h(X)\in R[X]$ be such that its image $\bar h(X)$ in $R/p_{\Psi_F}[X]$ is not zero. Then for all $i$ large enough we have $\nu(h(\sigma_i))\in\Psi_F$, and the following are equivalent:
\begin{enumerate}
\item\label{enum:quot1} The sequence $(\nu(h(\sigma_i)))_{i\geq1}$ is not eventually constant.
\item\label{enum:quot2} The image of $\bar h(X)$ in $S_F$ belongs to $H_{S_F}(\bar\nu_F)$.
\end{enumerate}
\end{lemma}
\begin{proof} 
This is automatically true if $\bar h(X)$ is a non zero constant polynomial, so assume that $\deg\bar h(X)>0$. The elements $\bar h(\bar\sigma_i)\in R/p_{\Psi_F}$ are the images of the $h(\sigma_i)$ under the natural epimorphism $R\to R/p_{\Psi_F}$, and they cannot be zero for infinitely many values of $i$, so that $h(\sigma_i)\notin p_{\Psi_F}$, which means that $\nu(h(\sigma_i))=\bar\nu_F(\bar h(\bar\sigma_i))\in\Psi_F$, at least for large $i$. This equality proves the first part of the result and will be implicitly used in what follows.

Let us now prove that \eqref{enum:quot1} and \eqref{enum:quot2} are equivalent. Recall that \eqref{enum:quot2} holds if and only if $\bar h(\bar\sigma_\infty)= 0$ in $L_F(\bar\sigma_\infty)$. In view of Corollary~\ref{approx}, if $\nu(h(\sigma_i))$ is constant for large $i$, its value is the $\tilde{\bar\nu}_F$-value of the element $\bar h(\bar\sigma_\infty)$ and is in $\Psi_F$, therefore $\bar h(\bar\sigma_\infty)\neq 0$. If $\nu(h(\sigma_i))$ is not constant for large $i$, by Corollary~\ref{approx} there exists $k$ such that $1\leq k\leq\deg\bar h(X)$ and
\[\tilde{\bar\nu}_F(\bar h(\bar\sigma_\infty))>\bar\nu_F(\bar h(\bar\sigma_i))=\tilde{\bar\nu}_F(\bar h_k(\bar\sigma_\infty)(-1)^k\bar\eta_i^k)=\tilde{\bar\nu}_F(\bar h_k(\bar\sigma_\infty))+k\,\bar\nu_F(\bar\delta_i)\] 
for all $i$ large enough, where $\bar\eta_i=\bar\sigma_\infty-\bar\sigma_i$ and $\bar h_k(X)\in R/p_{\Psi_F}[X]$. These inequalities and Lemma~\ref{cofinal} imply that the $\bar\nu_F(\bar h(\bar\sigma_i))$ are cofinal in the value group of $\tilde{\bar\nu}_F$. Thus, $\bar h(\bar\sigma_\infty)=0$.\end{proof}

\begin{remark}\label{partialanswer}
We know how to compute the $\tilde\nu$-value of any non zero element of $K(\si)$ once we know to calculate $\tilde\nu(h(\si))$ for $h(X)\in R[X]$ such that $0<\deg h(X)<\deg F^*(X)$. Let us apply Lemma~\ref{quot} and Corollary~\ref{approx} to such a polynomial. When $\Psi_F=\Phi$, these results say that $\tilde\nu(h(\si))=\nu(h(\sigma_i))$ for all $i$ large enough. If $\Psi_F\subsetneq\Phi$, then we would reach the same conclusion if could guarantee that $\bar h(X)\in R/p_{\Psi_F}[X]$ is a non zero polynomial whose degree is less that the degree of the minimal polynomial of $\bar\sigma_\infty$ over $L_F$. This condition on the degree is satisfied if $F(X)$ becomes irreducible, more precisely, when its image $\widebar F(X)\in R/p_{\Psi_F}[X]$ is irreducible in $L_F[X]$. Observe that $\deg \bar h(X)\leq\deg h(X)<\deg F^*(X)\leq \deg F(X)=\deg \widebar F(X)$.
\end{remark}

As explained in Remark~\ref{changeofvariable}, given $\alpha\in m_R$, $F_\alpha(X'):=F(X'+\alpha)\in R[X']$ is a Nagata polynomial defining a Nagata extension $S_\alpha$ of $R$ that is isomorphic to $S$. Moreover, since $\nu(F^*(\alpha))>0$ we have $F_\alpha^*(X')=F^*(X'+\alpha)$, and therefore $S/H_S(\nu)$ is isomorphic to $S_\alpha/H_{S_\alpha}(\nu)$.\par

So we may first make a change of variable $X'=X-\alpha$ with $\alpha\in m_R$, and then construct the 
Newton sequence of values $(\delta_i^{(\alpha)})_{i\in\N}$ and the sequence of partial sums $(\sigma_i^{(\alpha)})_{i\geq1}$ attached to the polynomial $F_\alpha(X')$ by iterating Newton's method.\par

Since the Nagata extension $R\to S$ is non trivial, $F(a)\neq 0$ for all $a\in m_R$, and the set $\set{\nu(F(a))\,\vert\,a\in m_R}$ is contained in $\Phi$. This subset of the value group does not depend on the choice of the variable: for all $\alpha\in m_R$, it equals the set  $\set{\nu(F_\alpha(a))\,\vert\,a\in m_R}\subseteq\Phi$.\par\noindent
\begin{definition}
Let $\nu$ be a valuation centered in a local domain $R$ and let $\Phi$ be its value group. Let $F(X)\in R[X]$ be a Nagata polynomial defining a non trivial Nagata extension of $R$. The \emph{intrinsic convex subgroup $\Psi$ of $\Phi$ associated to $F(X)$} is the smallest convex subgroup of $\Phi$ containing all the $\nu(F(a))$ with $a\in m_R$.
\end{definition}

Given $\alpha\in m_R$, in Definition~\ref{def:noni} we associated to the polynomial $F_\alpha(X')\in R[X']$ a convex subgroup $\Psi_{F_\alpha}$ of $\Phi$. Next we explain the relationship between the intrinsic convex subgroup $\Psi$ and these $\Psi_{F_\alpha}$.\par\noindent
\begin{lemma}\label{intbigger}
With the notations of this subsection, we have:
\begin{enumerate}
\item\label{enum:incl}$\Psi_{F_\alpha}\subseteq\Psi$.
\item\label{enum:eq} $\Psi_{F_\alpha}=\Psi$ if and only if $\overline{F_\alpha}(X')\in R/p_{\Psi_{F_\alpha}}[X']$ defines a non trivial Nagata extension of the local domain $R/p_{\Psi_{F_\alpha}}$.
\end{enumerate}
\end{lemma}
\begin{proof}
Statement \eqref{enum:incl} follows from the identities $\nu(\delta_0^{(\alpha)})=\nu(F(\alpha))$ and $\nu(\delta_i^{(\alpha)})=\nu(F(\sigma_i^{(\alpha)}+\alpha))$ for $i\geq1$. Having a strict inclusion $\Psi_{F_\alpha}\subsetneq\Psi$ is equivalent to the existence of $a\in m_R$ such that $F_\alpha(a)\in p_{\Psi_{F_\alpha}}$. In turn, this is equivalent to saying that the polynomial $\overline{F_\alpha}(X')\in R/p_{\Psi_{F_\alpha}}[X']$ has a root in $m_R/p_{\Psi_{F_\alpha}}$, which means that it defines a trivial Nagata extension of $R/p_{\Psi_{F_\alpha}}$.\end{proof}

\begin{definition}\label{def:goodvarible}
We say that $X'=X-\alpha$, where $\alpha\in m_R$, is a \emph{good variable} if the following conditions are satisfied:
\begin{enumerate}
\item\label{enum:verygood} $\Psi_{F_\alpha}=\Psi$.
\item\label{enum:firstvalue} $\Psi_{F_\alpha}$ is the smallest convex subgroup of $\Phi$ containing $\nu(\delta_0^{(\alpha)})$.
\end{enumerate}
\end{definition}

In the finite rank case, we can assume that our Nagata extension is of the form $S_\alpha$ with $\alpha\in m_R$ defining a good variable:

\begin{lemma}\label{good}
If the rank of $\nu$ is finite then there exists a good variable.
\end{lemma}
\begin{proof}
We first concentrate on \eqref{enum:verygood} in Definition~\ref{def:goodvarible}. If it is satisfied for $\alpha=0$ then we are done. Assume that this is not the case. Then we can choose $a_1\in m_R$ such that $F(a_1)\in p_{\Psi_F}$. Set $\alpha_1=a_1$ and consider $F_{\alpha_1}(X')=F(X'+a_1)$. Since $\nu(\delta_0^{(\alpha_1)})=\nu(F(a_1))\notin\Psi_F$ and the convex subgroups of $\Phi$ are totally ordered by inclusion, it follows that $\Psi_F\subsetneq\Psi_{F_{\alpha_1}}$. If $\Psi_{F_{\alpha_1}}=\Phi$, then $\alpha_1$ satisfies \eqref{enum:verygood} and we stop. Otherwise we pick $a_2\in m_R$ such that $F_{\alpha_1}(a_2)$ belongs to $p_{\Psi_{F_{\alpha_1}}}$ and set $\alpha_2=\alpha_1+a_2$. Now consider the polynomial $F_{\alpha_2}(X')=F(X'+\alpha_2)$. We have $\nu(\delta_0^{(\alpha_2)})=\nu(F_{\alpha_1}(a_2))\notin\Psi_{F_{\alpha_1}}$ and, by the same argument as before, we get a chain $\Psi_F\subsetneq\Psi_{F_{\alpha_1}}\subsetneq\Psi_{F_{\alpha_2}}\subseteq\Phi$. After a finite number of iterations this process has to stop  because our assumption on its rank. But it cannot stop unless $\Psi_{F_\alpha}=\Psi$.

Suppose that $\Psi_{F}=\Psi$ but \eqref{enum:firstvalue} in Definition~\ref{def:goodvarible} is not satisfied. By Lemma~\ref{simple}, there exists $i_0>0$ such that $\Psi_{F}$ is the smallest convex subgroup of $\Phi$ containing $\nu(\delta_{i_0})$. Then $X'=X-\sigma_{i_0}$ is a good variable.\end{proof}

\begin{remark}
The assumption on the rank of the valuation is needed to guarantee that condition \eqref{enum:verygood} in Definition~\ref{def:goodvarible} can be achieved.
\end{remark}

\subsection{End of the proof of Theorem 1}\label{subsec:proof} Finally, we prove the main result of this section and obtain Theorem~\ref{mainthm}.\eqref{grouphext} as a corollary of it.\par\noindent
\begin{proposition}\label{computationNagata}
Let $\nu$ be a valuation centered in a local domain $R$ with value group $\Phi$. Let $R\to S$ be a Nagata extension and let $\tilde\nu$ be the unique valuation centered in $S/H_S(\nu)$ extending $\nu$ through the inclusion $R\subset S/H_S(\nu)$. Then, the value group of $\tilde\nu$ is $\Phi$.
\end{proposition}

\begin{proof}
Let $F(X)$ be a Nagata polynomial of degree $n$ defining the extension $R\to S$, which we assume to be non trivial.

\textit{Reduction to the case where $R$ is integrally closed:} Keep the notation introduced before Lemma~\ref{intclosure}. In view of the following commutative diagram (see \eqref{diagram1} in Remark~\ref{firstsimpl}),
\[\xymatrixcolsep{2pc}
\xymatrix@R-1pc{
R \ar@{^{(}->}[d] \ar@{^{(}->}[r] & S/H_S(\nu) \ar@{^{(}->}[d] \ar[r]^-{\simeq}& R[\si]_* \ar@{^{(}->}[d] \ar@{^{(}->}[r]& K(\si)\\
\tilde R \ar@{^{(}->}[r] & \tilde S \ar[r]^-{\simeq}& \tilde R[\si]_* \ar@{^{(}->}[ru]}
\]
where all the vertical maps are local ring homomorphisms, it suffices to show the result for the Nagata extension $\tilde R\to\tilde S$ and the same valuations $\nu$ and $\tilde\nu$.

\textit{From now on, $R$ is an integrally closed local domain.} By Lemma~\ref{intclosure} and our assumption on $R$, we can assume that $F(X)=F^*(X)$, which simplifies the notation in what follows. Any element of the value group of $\tilde\nu$ is of the form $\tilde\nu(a\,h(\si))$ where $a\in R\setminus\set{0}$ and $h(X)$ is a polynomial in $R[X]$ such that $0\leq\deg h(X)<n$. Therefore it is enough to prove that $\tilde\nu(h(\si))$ belongs to $\Phi$. If its degree is zero, then $\tilde\nu(h(\si))=\nu(h(\si))\in\Phi$; thus we may assume that $\deg h(X)>0$.

\textit{Reduction to case where $\nu$ is of finite rank:} Let $P_0$ be the prime ring in $R$, that is, $\Z/p\Z$ if $R$ is of characteristic $p$, and $\Z$ otherwise. The ring $R$ is the inductive limit of its $P_0$--subalgebras that are essentially of finite type. Since $R$ is integrally closed, one may restrict the inductive system to the integrally closed local subalgebras dominated by $R$ that are essentially of finite type over $P_0$. This is because both $\Z$ and $\mathbf{F}_p$ are Nagata rings (see \cite[Ch.~14, no.1]{L-M}, \cite[Chap. 12, p.241]{Mat},) so that the integral closure of such a subalgebra in its field of fractions is noetherian and its localization at the center of the valuation is again noetherian.  Let us consider a subalgebra $R_0$ of $R$ of this type and containing the coefficients of the polynomial $F(X)$. We call $K_0\subseteq K$ the fraction field of $R_0$.

The valuation $\nu$ induces a valuation $\nu_0$ centered in $R_0$ whose valuation ring is $R_{\nu_0}=R_\nu\cap K_0$. Since $R_0$ is essentially of finite type over $P_0$, the value group $\Phi_0$ of $\nu_0$ has finite rational rank. Indeed, this rank is bounded by the transcendence degree of $K_0$ over the fraction field of $P_0$ by Abhyankar's inequality (see \cite[Ch.~VI, \S~10, no.3, Cor.1]{B}). Therefore $\Phi_0$ has finite rank.

The polynomial $F(X)\in R_0[X]$ is a Nagata polynomial and, if $S_0$ is the Nagata extension of $R_0$ defined by $F(X)$ and  $\tilde\nu_0$ is the extension of $\nu_0$ to $S_0$ (note that $H_{S_0}(\nu_0)=(0)$, see Remark~\ref{firstsimpl}), we have that $R_{\tilde\nu_0}$ corresponds to $R_{\tilde\nu}\cap K_0(\si)$ and a commutative diagram as follows:
\[\xymatrix@R-1pc{
R_0 \ar@{^{(}->}[d] \ar@{^{(}->}[r] & S_0 \ar@{^{(}->}[d] \ar[r]^-{\simeq}& R_0[\si]_* \ar@{^{(}->}[d] \ar@{^{(}->}[r] & K_0(\si) \ar@{^{(}->}[d] \\
R \ar@{^{(}->}[r] & S \ar[r]^-{\simeq}& R[\si]_* \ar@{^{(}->}[r] & K(\si)}
\]
where all the vertical maps are local ring homomorphisms. Since given a finite number of polynomials in $R[X]$ one may assume that $R_0$ contains the coefficients of all the polynomials, we see that it is enough to prove the result in the case where the local domain is integrally closed and the rank of the valuation is finite.

\textit{Proof in the case where $R$ is integrally closed and $\nu$ is of finite rank:} According to Lemma~\ref{quot} and Corollary~\ref{approx}, if $\Psi_F=\Phi$ then $\tilde\nu(h(\si))=\nu(h(\sigma_i))$ for all $i$ large enough, and this ends the proof. Next we treat the case $\Psi_F\subsetneq\Phi$.

The polynomial $F(X)$ defines non trivial Nagata extensions of $R$ and $R_\nu$ that sit in a natural commutative diagram of local ring homomorphisms:
\begin{equation} \tag{$\text{D}_2$}
\begin{gathered}
\xymatrix@R-1pc{
R \ar@{^{(}->}[d] \ar@{^{(}->}[r] & S=R[\si]_* \ar@{^{(}->}[d] \ar@{^{(}->}[r]& R_{\tilde\nu}\\
R_\nu \ar@{^{(}->}[r] & R_\nu[\si]_* \ar@{^{(}->}[ru]&}
\end{gathered}
\label{diagram2}
\end{equation}

Every finitely generated ideal of $R_\nu$ is principal and generated by an element of its set of generators. Therefore $h(X)$ can be written in $R_\nu[X]$ as $h(X)=h_t H(X)$ with $h_t\in R_\nu\setminus\set{0}$ and $H(X)\in R_\nu[X]$ having a coefficient equal to one. 

Let $\bar\nu:=\bar\nu_{\Psi_F}$ be the residual valuation (corresponding to $\Psi_F$) with which $\nu$ is composed and let $L$ be the fraction field of its valuation ring $R_{\bar\nu}=R_\nu/m_{\Psi_F}$. Fix an algebraic closure of $L$ and write the Nagata extension defined by the image $\widebar F(X)$ of $F(X)$ in $R_{\bar\nu}[X]$ as $R_{\bar\nu}\hookrightarrow R_{\bar\nu}[\bar\sigma_\infty]_*$, where $\widebar F(\bar\sigma_\infty)=0$.
Let us first assume that $\widebar F(X)\in R_{\bar\nu}[X]$ is irreducible. Since $R_{\bar\nu}$ is integrally closed (recall that any valuation ring has this property, see \cite[Ch.~VI, \S~1, no. 3, Corollary 1]{B}) and the polynomial is monic, this is equivalent to being irreducible in $L[X]$. Therefore $\widebar F(X)$ is the minimal polynomial of $\bar\sigma_\infty$ over $L$. Then we finish the proof as follows:

On the one hand, the image $\widebar H(X)$ of $H(X)$ in $R_{\bar\nu}[X]$ is not zero because at least one of its coefficients equals $1$. On the other hand, we have $\deg\widebar H(X)\leq\deg h(X)<n=\deg\widebar F(X)$ and thus $\widebar H(\bar\sigma_\infty)\neq 0$. Applying Lemma~\ref{quot} and Corollary~\ref{approx} (to $R_\nu$, $F(X)$, and $H(X)$) we conclude that $\tilde\nu(H(\si))=\nu(H(\sigma_i))$ for all $i$ large enough. This fact and the relation between $h(X)$ and $H(X)$ show that, for large $i$, $\tilde\nu(h(\si))=\nu(h(\sigma_i))$.

Next we address the proof in the case where $\widebar F(X)$ is a multiple of the minimal polynomial of $\bar\sigma_\infty$ over $L$. This minimal polynomial has coefficients in $R_{\bar\nu}$ by Lemma~\ref{intclosure}.\eqref{enum:defminpol}. Let us write it as the image $\widebar Q(X)$ in $R_{\bar\nu}[X]$ of a monic polynomial $Q(X)\in R_\nu[X]$ with $\deg Q(X)=\deg\widebar Q(X)<n$. We have $\bar\nu(\widebar Q(0))>0$ and an equality $\widebar F(X)=\widebar G(X)\widebar Q(X)$ in $R_{\bar\nu}[X]$. Hence by Lemma~\ref{fac}, $\widebar Q(X)$ is a Nagata polynomial and $\widebar G(0)$ does not belong to the maximal ideal $m_{\bar\nu}=m_\nu/m_{\Psi_F}$ of $R_{\bar\nu}$. It is straightforward that $Q(X)\in R_\nu[X]$ is an irreducible Nagata polynomial. 

Let $(\epsilon_i)_{i\in\N}$ and $(\tau_i)_{i\geq1}$ be the Newton sequence of values and the sequence of partial sums attached to $Q(X)$. We can extend $\tilde\nu$ to the algebraic closure $\widebar K$ of $K$ and therefore consider the distinguished root $\tau^{(0)}_\infty\in\widebar K$ of $Q(X)$ satisfying $\tilde\nu(\tau^{(0)}_\infty)=\nu(\epsilon_0)>0$. Note that $Q(X)$ is the minimal polynomial of $\tau^{(0)}_\infty$. In addition, by Euclidean division (in the rings $R_\nu[X]$ and $R_{\bar\nu}[X]$),
\begin{equation}\label{B}
F(X)=G_1(X)Q(X)+B(X) 
\end{equation}
with $\widebar G_1(X)=\widebar G(X)$ and $\widebar B(X)=0$ in $R_{\bar\nu}[X]$, and thus $G_1(0)\notin m_\nu$ and $B(X)\in m_{\Psi_F}R_\nu[X]$. 

Since $\nu$ has finite rank, by Lemma~\ref{good} we can assume that $X$ is a good coordinate. Under this assumption we have $\Psi_F=\Psi$, and hence $\widebar Q(X)$ and $Q(X)$ both define non trivial Nagata extensions. Indeed, $\widebar Q(X)$ defines the same Nagata extension as $\widebar F(X)$ and, if $Q(a)=0$ for some $a\in m_R$, then $\nu(F(a))=\nu(B(a))\notin\Psi_F$, which is a contradiction.

Let $\Psi_Q$ (resp. $\Psi'$) be the convex subgroup (resp. the intrinsic convex subgroup) of $\Phi$ associated to $Q(X)$. Given $a\in m_R$, evaluating in $a$ the expression \eqref{B} above, and taking into account that $\nu(G_1(a))=0$ and that $\nu(B(a))$ is greater than any element in $\Psi_F$, we obtain that $\nu(F(a))=\nu(Q(a))$. On the one hand, by definition of $\Psi'$ and Lemma~\ref{intbigger}.\eqref{enum:incl}, we have  $\Psi_Q\subseteq\Psi'=\Psi=\Psi_F$. On the other hand, taking $a=0$ gives that $\nu(\delta_0)=\nu(\epsilon_0)$, and since $\Psi_F$ is the smallest convex subgroup containing $\nu(\delta_0)$, we get $\Psi_F\subseteq\Psi_Q$. This shows that $\Psi_Q=\Psi'=\Psi=\Psi_F$.

We write the Nagata extension defined by $Q(X)\in R_\nu[X]$ as $R_\nu\hookrightarrow R_1=R_\nu[\tau^{(0)}_\infty]_*\subset K(\tau^{(0)}_{\infty})$. Let $\nu_1$ be the valuation centered in $R_1$ which extends $\nu$ and denote by $R_{\nu_1}$ its valuation ring. Note that $R_{\nu_1}$ coincides with $R_{\tilde\nu}\cap K(\tau^{(0)}_\infty)$. Moreover, the value group of $\nu_1$ is the group $\Phi$. Indeed, we have $\bar\nu_{\Psi_Q}=\bar\nu_{F}$ and $\widebar Q(X)\in R_{\bar\nu}[X]$ irreducible, and we have seen before that the result is true in this case. The key point is that $\Psi_Q=\Psi_F$.

Let us now consider the polynomial $F_1(X_1)=F(X_1+\tau^{(0)}_\infty)\in R_1[X_1]$. It is still a Nagata polynomial of degree $n$ and vanishes at $\si^{(1)}=\si-\tau^{(0)}_\infty$. In fact, $F(X)$ is a Nagata polynomial in $R_1[X]$ and $\tau^{(0)}_\infty\in m_{R_1}$, so with the notation of Subsection~\ref{subsec:firststeps}, $F_1(X_1)$ is the polynomial $F_{\alpha}(X_1)$ for $\alpha=\tau^{(0)}_\infty$, $X_1=X-\alpha$. In addition, $\tilde\nu$ takes a positive value on $\si^{(1)}$ because $\tilde\nu(\si)$ and $\tilde\nu(\tau^{(0)}_\infty)$ are both positive (more precisely, we have $\tilde\nu(\si)=\nu(\delta_0)=\nu(\epsilon_0)=\tilde\nu(\tau^{(0)}_\infty)\in\Psi_{F,>0}$). Since $R_1$ is an integrally closed local domain (it is a Nagata extension of an integrally closed domain, see \cite[Proposition 7]{L}), the polynomial $F_1(X_1)$ determines the Nagata extension 
\[R_1\hookrightarrow S_1=R_1[\si^{(1)}]_*\subset K(\si,\tau^{(0)}_\infty),\]
which contains $\si$. The extension $R_\nu[\si]_*\hookrightarrow S_1$ therefore is a Nagata extension and, setting $h_1(X_1)=h(X_1+\tau^{(0)}_\infty)$, the element $h(\si)$ is mapped to $h_1(\si^{(1)})$ which is non zero because $\deg h(X)<n$ and $h(\si)$ and $h(\si^{(1)}+\tau^{(0)}_\infty)$ have the same image in the henselization of $R_{\nu}$. The valuation $\widetilde{\nu_1}$ centered in $S_1$ extending $\nu_1$ has valuation ring $R_{\widetilde{\nu_1}}=R_{\tilde\nu}\cap K(\si,\tau^{(0)}_\infty)$.

Figure \ref{Fig1} might help the reader to visualize this construction: we have a commutative diagram of local ring homomorphisms, where the last horizontal arrow corresponds to the Nagata extension defined by $F_1(X_1)$ seen as Nagata polynomial in $R_{\nu_1}[X_1]$.
\begin{figure}[!htbp]
\[\xymatrix@R-1pc{
R \ar@{^{(}->}[d]\ar@{^{(}->}[r]      & S=R[\si]_* \ar@{^{(}->}[d]\ar@{^{(}->}[rd]  \\
R_\nu \ar@{^{(}->}[d]\ar@{^{(}->}[r]  & R_\nu[\si]_* \ar@{^{(}->}[d]\ar@{^{(}->}[r]  &R_{\tilde\nu}\ar@{^{(}->}[d]\\
R_1=R_\nu[\tau^{(0)}_\infty]_* \ar@{^{(}->}[d]\ar@{^{(}->}[r] & S_1=R_1[\si^{(1)}]_* \ar@{^{(}->}[d]\ar@{^{(}->}[r] &  R_{\widetilde{\nu_1}}\\
R_{\nu_1} \ar@{^{(}->}[r] & R_{\nu_1}[\si^{(1)}]_* \ar@{^{(}->}[ru]
}\]
\caption{Commutative diagram constructed after one step.}\label{Fig1}
\end{figure}

Let $(\delta^{(1)}_i)_{i\in\N}$ and $(\sigma^{(1)}_i)_{i\geq1}$ be the Newton sequence of values and the sequence of partial sums attached to $F_1(X_1)$, respectively. Let $\Psi_{F_1}$ be the convex subgroup of $\Phi$ associated to $F_1(X_1)$. Since $\nu_1(\delta^{(1)}_0)=\nu_1(F_1(0))=\nu_1(F(\tau^{(0)}_\infty))=\nu_1(B(\tau^{(0)}_\infty))\notin\Psi_F$ and the set of convex subgroups of $\Phi$ is totally ordered by inclusion, we have  $\Psi_F\subsetneq\Psi_{F_1}\subseteq\Phi$. Note that \[\nu_1(\delta^{(1)}_0)=\tilde\nu(\si^{(1)})>\tilde\nu(\si)=\tilde\nu(\tau^{(0)}_\infty).\]

If $\Psi_{F_1}=\Phi$ then Lemma~\ref{quot} and Corollary~\ref{approx} imply that,
\begin{equation}\label{onestep}
    \tilde\nu(h(\si))=\widetilde{\nu_1}(h_1(\si^{(1)}))=\nu_1(h_1(\sigma^{(1)}_i))=\nu_1(h(\sigma^{(1)}_i+\tau^{(0)}_\infty))\text{ for all }i\text{ large enough.}
\end{equation}

Assume that $\Psi_{F_1}\subsetneq\Phi$. Let $\widebar{\nu_1}:={\overline{\nu_1}}_{\Psi_{F_1}}$ be the residual valuation (corresponding to $\Psi_{F_1}$) with which $\nu_1$ is composed. Suppose that the image $\overline{F_1}(X_1)\in R_{\overline{\nu_1}}[X_1]$ of $F_1(X_1)$ is irreducible (which in turn implies that $F_1(X_1)$ is the minimal polynomial of $\si^{(1)}$ over $K(\tau_\infty^{(0)})$). Then, writing $h_1(X_1)=h_{1,t_1} H_1(X_1)$ in $R_{\nu_1}[X_1]$, where $h_{1,t_1}$ is a non zero constant and $H_1(X_1)$ has a coefficient equal to one, and repeating the same arguments as before, we conclude that the statement \eqref{onestep} above also holds in this case. If $\overline{F_1}(X_1)$ is reducible, we lift its Nagata factor $\overline{Q_1}(X_1)$ to an irreducible Nagata polynomial $Q_1(X_1)\in R_{\nu_1}[X_1]$, which has a unique root $\tau^{(1)}_\infty\in\widebar K$ of positive $\tilde\nu$-value, and repeat the construction.

Since $\nu_1$ has finite rank, we can assume that $X_1$ is a good variable. We produce a Nagata extension $R_{\nu_1}\hookrightarrow R_2=R_{\nu_1}[\tau^{(1)}_\infty]_*\subset K(\tau^{(0)}_\infty,\tau^{(1)}_\infty)$ such that again the valuation extends uniquely to a valuation $\nu_2$ centered in $R_2$ and the value group does not change (so it is equal to $\Phi$). The Nagata polynomial $F_2(X_2)=F_1(X_2+\tau_\infty^{(1)})\in R_2[X_2]$ determines a Nagata extension \[R_2\hookrightarrow S_2=R_2[\si^{(2)}]_*\subset K(\si,\tau^{(0)}_\infty,\tau^{(1)}_\infty)\] associated to the root of positive $\tilde\nu$-value $\si^{(2)}=\si^{(1)}-\tau^{(1)}_\infty=\si-\sum_{k=0}^1{\tau^{(k)}_\infty}$, and which contains $R_{\nu_1}[\si^{(1)}]_*$. We call $\widetilde{\nu_2}$ the valuation centered in $S_2$ extending $\nu_2$, $R_{\widetilde{\nu_2}}=R_{\tilde\nu}\cap K(\si,\tau^{(0)}_\infty,\tau^{(1)}_\infty)$. The element $h(\sigma_\infty)$ is written in $S_2$ as the non zero element $h_2(\si^{(2)})=h(\si^{(2)}+\sum_{k=0}^1{\tau^{(k)}_\infty})$. In addition, we have $\Psi_F\subsetneq\Psi_{F_1}\subsetneq\Psi_{F_2}\subseteq\Phi$ and $\nu_2(\delta_0^{(2)})=\tilde\nu(\si^{(2)})>\tilde\nu(\si^{(1)})=\tilde\nu(\tau_\infty^{(1)})$.

We now ask whether $\Psi_{F_2}=\Phi$, or $\Psi_{F_2}\subsetneq\Phi$ and we get an irreducible polynomial after reducing the coefficients of $F_2(X_2)$ modulo the prime ideal of $R_{\nu_2}$ corresponding to $\Psi_{F_2}$. If one of these conditions holds then, for all $i$ large enough, \[\tilde\nu(h(\si))=\widetilde{\nu_2}(h_2(\si^{(2)}))=\nu_2(h_2(\sigma^{(2)}_i))=\nu_2(h(\sigma^{(2)}_i+\sum_{k=0}^1{\tau^{(k)}_\infty});\] otherwise, we repeat the construction.

As we iterate this construction, the value groups of the valuations $\set{\nu_k}_{k\geq1}$ that we create remain equal to $\Phi$ which has a bounded rank, and the convex subgroups $\set{\Psi_{F_k}}_{k\geq1}$ that we determine grow strictly, so this process has to stop after finitely many steps, say $\ell\geq1$ steps. This proves that $\tilde\nu(h(\sigma_\infty))$ is the value of the image $h_\ell(\si^{(\ell)})=h(\si^{(\ell)}+\sum_{k=0}^{\ell-1}{\tau^{(k)}_\infty})$ of $h(\sigma_\infty)$ in a finite extension of $K$, for a uniquely defined extension $\tilde\nu_\ell$ of $\nu$; and it is also the $\nu_\ell$-value for large $i$ of $h(\sigma_i^{(\ell)}+\sum_{k=0}^{\ell-1}{\tau^{(k)}_\infty})$. The value group is preserved since by construction the value group of $\nu_\ell$ is the same as that of $\nu$.\end{proof}

The proof of Theorem~\ref{mainthm}.\eqref{grouphext} is now straightforward.\par\noindent
\begin{proof}[Proof of Theorem~\ref{mainthm}.\eqref{grouphext}]
Since the valuation $\nu$ extends uniquely to each $S/H_S(\nu)$ without changing the value group by Proposition~\ref{computationNagata}, the same is true for $R^h/H(\nu)=\bigcup_S{S/H_S(\nu)}$.\end{proof}

\begin{remark} If it was infinite, the sequence $\tau^{(0)}_\infty+\tau_\infty^{(1)}+\tau_\infty^{(2)}+\cdots+\tau_\infty^{(k)}$ built in the proof of Proposition~\ref{computationNagata} would be a pseudo--convergent sequence of elements of the maximal ideal of $R_{\nu}^h$ for the extension $\tilde\nu$ of $\nu$, which would have the property that for each integer $j$ the smallest convex subgroup containing $\tilde\nu (\tau_\infty^{(j+1)})$ is strictly larger than the one containing $\tilde\nu (\tau_\infty^{(j)})$. Let us say that such a sequence is \textit{pseudo--convergent in scales}. What we use here is that by Abhyankar's inequality (see \cite[Ch.~VI, \S~10, no. 3, Corollary 1]{B}) there can be no infinite such sequence in a field of bounded transcendence degree over a valued field with a value group of bounded rational rank, in our case the prime field with the trivial valuation. By construction, $\sigma_\infty$ is in fact represented in $S_\ell$ by a finite sum $\sigma_\infty= \tau^{(0)}_\infty+\tau_\infty^{(1)}+\tau_\infty^{(2)}+\cdots+\tau_\infty^{(\ell-1)}+\sigma_\infty^{(\ell)}$.
\par\noindent  Here and at the end of this paper we have only used the fact that the value group has finite rank. A class of rings which have this property for any valuation and are not all noetherian are the rings of \textit{finite valuative dimension} in the sense of Jaffard in \cite[Ch.~IV]{J}.
\end{remark}

We may summarize the conclusion of the preceding discussion as follows:\par\noindent
\begin{definition}\label{nu-irred}
Let $\nu$ be a valuation centered in a local domain $R$ and let $F(X)\in R[X]$ be a Nagata polynomial. Let $\Psi_F$ be the convex subgroup of the value group of $\nu$ attached to $F(X)$ as in Definition~\ref{def:noni} and let $p_{\Psi_F}$ be the corresponding prime ideal of $R$. We say that $F(X)$ is \textit{$\nu$-residually irreducible} if the image $\widebar F(X)\in R/p_{\Psi_F}[X]$ of $F(X)$ is irreducible in $L_F[X]$, where $L_F$ denotes the fraction field of $R/p_{\Psi_F}$. 
\end{definition}

This irreducibility implies that $\widebar F(X)$ defines a non trivial Nagata extension of $R/p_{\Psi_F}$ and by Lemma~\ref{intbigger} that $\Psi_F$ is the intrinsic convex subgroup of the Nagata extension defined by $F(X)$. Remembering the discussion in Remark \ref{partialanswer} and the factorization $h(X)=h_tH(X)$ used in the proof of Proposition \ref{computationNagata}, it also implies that for any polynomial $h(X)\in R[X]$ such that $0\leq\deg h(X)<\deg F^*(X)$, we have for large $i$ the equality ${\rm in}_\nu(h(\sigma_i))={\rm in}_{\tilde\nu}(h(\sigma_\infty))$.\par

The algorithm described above has the following consequence:\par\noindent
\begin{proposition}\label{nu-ext}
Let $\nu$ be a valuation of finite rank centered in an integrally closed local domain $R$. Given a Nagata extension $R\hookrightarrow R[\sigma_\infty]_*$ corresponding to a Nagata polynomial $F(X)\in R[X]$, there exist a local domain $R'$ dominating $R$, dominated by the henselization of $R_\nu$ and containing $\sigma_\infty$, to which the valuation $\nu$ extends uniquely to a valuation centered in $R'$ and with the same value group, and an element $a'\in R'$ such that the Nagata polynomial $F'(X)=F(X+a')\in R'[X]$ is $\tilde\nu$-residually irreducible.
\end{proposition}
\begin{remark}In \cite[Theorem 3]{K}, Kaplansky shows that a pseudo-convergent sequence $(y_\tau)_{\tau\in T}$ of elements of a valued field $(K,\nu)$ which is of algebraic type (see section \ref{sec4} below for the definition) and has no limit in $K$ defines an immediate algebraic extension by the adjunction of a root of a polynomial $Q(X)\in K[X]$ of minimal degree among those for which $\nu(h(y_\tau))$ is not eventually stationary. In our situation the issue is that a Nagata polynomial $F(X)$ is a polynomial of minimal degree attached to the pseudo-convergent sequence $(\sigma_i)_{i\in\N}$ \textit{only if it is $\nu$-residually irreducible in the sense of definition \ref{nu-irred}}. To overcome this difficulty we show the existence of a sequence of Nagata extensions $R^{(j)}=R[\sigma_\infty, \tau^{(0)}_\infty,\tau_\infty^{(1)},\tau_\infty^{(2)},\cdots ,\tau_\infty^{(j)}]_{m_j}$ of $R$ by Nagata polynomials which do satisfy the $\nu$-residual irreductiblity condition and thus provide well defined immediate extensions $(R^{(j)},\nu^{(j)})$. The construction stops only when  $F(X+\tau^{(0)}_\infty+\tau_\infty^{(1)}+\tau_\infty^{(2)}+\cdots+\tau_\infty^{(j)})$ is $\nu^{(j)}$-residually irreducible, and we prove that it has to stop. 
\end{remark}
\section{Applications}
\subsection{Connected components of the Riemann--Zariski space}\label{RZ} In this subsection we apply Theorem~\ref{mainthm}.\eqref{hext} in the study of the Riemann--Zariski space $\RZ{(R)}$ of valuations centered in a local domain $R$.\par

We start with the following result which, in a slightly different formulation, is classical (see \cite[Chap. III, \S 17]{E}, \cite[Theorem 5.14]{Ku2}). See also  \cite[\href{https://stacks.math.columbia.edu/tag/0ASF}{Tag 0ASF},  Lemma 15.103.5]{stacks-project}\footnote{We are grateful to Laurent Moret-Bailly for this reference.}. We give a proof in the spirit of this paper:\par\noindent
\begin{corollary}\label{V} The henselization of a valuation ring is a valuation ring with the same value group.
\end{corollary}
\begin{proof}
A valuation ring $R_\nu$ is integrally closed, so the henselization $R^h_\nu$ of $R_\nu$ is a local domain (and there is no minimal prime to consider, see Remark~\ref{firstsimpl}) and it is integrally closed (see \cite[Theorem 13.12]{L-M} or \cite[Theorem 18.6.9]{EGA}). As such, according to \cite[Ch.~VI, \S~1, no. 3, Theorem 3]{B}, $R^h_\nu$ is the intersection of all the valuation rings of $K^h$ which dominate it, where $K^h$ is the fraction field of $R^h_\nu$. By the uniqueness of the extension of the valuation $\nu$ (Theorem~\ref{mainthm}.\eqref{hext}), among all the valuation rings of $K^h$ which dominate $R_\nu$ there is only one which also dominates $R_\nu^h$, so $R^h_\nu$ is this valuation ring. The fact that the value group is the same follows from Theorem~\ref{mainthm}.\eqref{grouphext}.\end{proof}
\begin{remark}\label{hh}As we mentioned in the abstract, corollary \ref{V} implies our theorem \ref{mainthm}: by functoriality the inclusion $R\subset R_\nu$ extends uniquely to a map $R^h\rightarrow R_\nu^h$ whose kernel, since its intersection with $R$ is zero and $R^h$ is integral over $R$, is a minimal prime of $R^h$. If we know that $R_\nu^h$ is a valuation ring with the same value group as $R_\nu$ this gives theorem \ref{mainthm}.
\end{remark}
Let $R$ be a local domain and let $\RZ{(R)}$ be the space of valuations centered in $R$. If $K$ is the fraction field of $R$, then $\RZ{(R)}$ consist of the set of all valuation rings of $K$ which dominate $R$ endowed with the Zariski topology (see \cite[Ch.~VI, \S~17]{ZS}). This topology is obtained by taking as a basis of open sets the subsets $U(A)$, whose elements are the valuation rings of $K$ dominating $R$ and containing $A$, where $A$ ranges over the family of all finite subsets of $K$.\par\noindent
\begin{corollary}\label{homeoRZ}
Let $R$ be a local domain and let $\{H_\iota\}_{\iota\in I}$ be the set of minimal primes of $R^h$. Let $\varphi:\RZ{(R)}\to\bigsqcup_{\iota\in I}\RZ(R^h/H_\iota)$ be the map which to a valuation ring $R_\nu\in\RZ{(R)}$ associates the minimal prime $H(\nu)$ of $R^h$ and the valuation ring $R_{\tilde\nu}\in\RZ{(R^h/H(\nu))}$ of the extension $\tilde\nu$ of $\nu$ to $R^h/H(\nu)$. Then, the map $\varphi$ satisfies the following:
\begin{enumerate}
    \item It is a homeomorphism.
    \item It induces a bijection between the set of connected components of $\RZ{(R)}$ and $\{H_\iota\}_{\iota\in I}$.
\end{enumerate}
\end{corollary}

\begin{proof}
Let $K$ be the fraction field of $R$ and let $\widebar R$ be the the integral closure of $R$ in $K$. Any maximal ideal of $\widebar R$ is of the form $m_\nu\cap\widebar R$ for some valuation ring $R_\nu$ of $K$ dominating $R$ (see \cite[Ch.~VI, \S~1, no. 3, Theorem 3]{B}). The ideal $H(\nu)$ of $R^h$ associated to $\nu$ then appears as the kernel of the canonical map $R^h\rightarrow {\tilde R}^h$, where $\tilde R=\widebar R_{m_\nu\cap\widebar R}$. Indeed, taking inductive limits of Nagata extensions in the commutative diagram obtained by combination of \eqref{diagram1} and the diagram \eqref{diagram2} for $\tilde R$, yields that $R^h\to R_\nu^h$ can be written as the injection ${\tilde R}^h\hookrightarrow R_\nu^h$ composed with $R^h\to{\tilde R}^h$. This defines a map from the set of maximal ideals of $\widebar R$ to $\{H_\iota\}_{\iota\in I}$. We now define its inverse map.

As in Remarks~\ref{rem:minprimes}.\eqref{enum:flat}, the fact that $R^h$ is flat over $R$ implies that for each $\iota\in I$ we have $H_\iota\cap R=(0)$. Since the natural composed map $p_\iota:R\to R^h\to R^h/H_\iota$ is injective, 
it induces an injection $\widebar R\hookrightarrow\widebar{R^h/H_\iota}$, where again the bar means integral closure. The maximal ideal of the domain $\widebar{R^h/H_\iota}$, which is local by \cite[Ch.~IX, Corollaire 1]{R}, induces a maximal ideal, say ${\widebar m}_\iota$, of $\widebar R$. We associate to $H_\iota$ this ${\widebar m}_\iota$. Denoting by $\widebar R_{\widebar m_\iota}$ the localization of $\widebar R$ at $\widebar m_\iota$, we have injections $R\hookrightarrow\widebar{R}_{{\widebar m}_\iota}\hookrightarrow\widebar{R^h/H_\iota}$ and a commutative diagram of local ring maps considering these maps and $R^h/H_\iota\hookrightarrow\widebar{R^h/H_\iota}$ composed with $p_\iota$. Since the integral closure of a henselian local domain is a henselian local domain as an inductive limit of finite algebras (see \cite[Ch.~IV, \S~18, Theorem 18.5.11 and Proposition 18.6.14]{EGA}), $\widebar{R^h/H_\iota}$ is henselian and by the universal property of henselization the natural map $R^h\to\widebar{R^h/H_\iota}$ factors uniquely through the map $(\widebar{R}_{{\widebar m}_\iota})^h\rightarrow\widebar{R^h/H_\iota}$. This map is injective because its kernel should have intersection zero with $\widebar R_{\widebar m_\iota}$ and therefore by \cite[Ch.~V, \S~2, no. 1, Corollary 1]{B} should be a minimal prime of the domain $(\widebar{R}_{{\widebar m}_\iota})^h$. This shows that $H_\iota$ is indeed the kernel of the map $R^h\rightarrow(\widebar R_{\widebar m_\iota})^h$. Observe also that we have local ring maps $R^h/H_\iota\hookrightarrow(\widebar R_{\widebar m_\iota})^h\hookrightarrow\widebar{R^h/H_\iota}$. Hence if $R_{\tilde\mu}\in\RZ{(R^h/H_\iota)}$, then $R_{\tilde\mu}$ dominates $\widebar{R^h/H_\iota}$ and the valuation ring $R_\mu=R_{\tilde\mu}\cap K$ of $K$ belongs to $\RZ{(\widebar R_{\widebar m_\iota})}$ and $m_\mu\cap\widebar R=\widebar m_\iota$ (since $\widebar m_\iota\overline R_{\widebar m_\iota}\cap\widebar R=\widebar m_\iota$).\par Thus, we have established a bijection between $\{H_\iota\}_{\iota\in I}$ and the set of maximal ideals of $\widebar R$, and in what follows we write this last set as $\{{\widebar m}_\iota\}_{\iota\in I}$ with $\widebar m_\iota$ corresponding to $H_\iota$.

As a set $\RZ{(R)}$ is the disjoint union of the family of subsets $\{\RZ{(\widebar R_{\widebar m_\iota})}\}_{\iota\in I}$. To prove that they are homeomorphic we observe that the Zariski topology on any $\RZ{(\widebar R_{\widebar m_\iota})}$ coincides with the topology induced by the topology of $\RZ{(R)}$. In addition, the local integral domains $\widebar R_{\widebar m_\iota}$ are unibranch and we can apply \cite[Theorem 2.4.2]{Tm} which ensures the connectedness of the spaces $\RZ{(\widebar R_{\widebar m_\iota})}$. So, in order to finish the proof, it suffices to take $\iota\in I$ and show that the bijective map $\varphi_\iota$ from 
$\RZ{(\widebar R_{\widebar m_\iota})}$ to $\RZ(R^h/H_\iota)$ induced by $\varphi$ is a homeomorphism.

On the one hand, the henselization $(\widebar R_{\widebar m_\iota})^h$ is a local domain (because $\widebar R_{\widebar m_\iota}$ is an integrally closed local domain, see Remark~\ref{firstsimpl}), and therefore by \cite[Proposition 3.4]{F}, the map from $\RZ{(\widebar R_{\widebar m_\iota})}$ to $\RZ((\widebar R_{\widebar m_\iota})^h)$ which sends a valuation to its unique extension is a homeomorphism. Note that in \cite[Proposition 3.4]{F} the noetherianity and excellence assumptions on the local ring are only needed to show that the previous map is well defined and bijective. On the other hand, as we have seen above, we have local ring maps $R^h/H_\iota\hookrightarrow(\widebar R_{\widebar m_\iota})^h\hookrightarrow\widebar{R^h/H_\iota}$ and hence $\RZ{((\widebar R_{\widebar m_\iota})^h)}$ coincides with $\RZ(R^h/H_\iota)$. These facts imply that $\varphi_\iota$ is a homeomorphism.\end{proof}

\subsection{Approximation of Henselian elements}\label{sec5}
In this subsection, instead of considering extensions of all the valuations centered in a local domain to its henselization, we study the extension of the valuation of a valuation ring to its henselization; we revisit a result of Franz-Viktor Kuhlmann in \cite[Theorem 1.1]{Ku}. This result concerns the approximation of elements of the henselization $(K^h,\tilde\nu)$ of a valued field $(K,\nu)$ by elements of $K$ and we can state it as follows since we know by Corollary~\ref{V} that $R^h_\nu=R_{\tilde\nu}$ and the value groups are equal:\par\noindent
\begin{theorem}{\rm (Kuhlmann)}\label{ap} Let $K$ be a field endowed with a valuation $\nu$ determined by the valuation ring $R_\nu$ and let $\Phi$ be the value group of $\nu$. Let $ K^h$ be the field of fractions of the henselization $R^h_\nu=R_{\tilde\nu}$ of $R_\nu$. For every element $z\in K^h\setminus K$ there exist a convex subgroup $\Psi$ of $\Phi$ and an element $\varphi\in\Phi$ such that $\varphi+\Psi$ is cofinal in the ordered set 
\[\tilde\nu (z-K)=\{\tilde\nu(z-c)\,\vert\,c\in K\}\subset\Phi.\]
\end{theorem}

Before giving the proof of this result, let us come back to the nature of the growth of the $\nu(h(\sigma_i))$, which is also a consequence of Corollary~\ref{approx}, but here we see directly that the coefficient of ${\rm in}_\nu\delta_i^e$ is the initial form of an element of $R_\nu$.\par\noindent
\begin{proposition}\label{ae} With the notations of Proposition~\ref{asym}, given a Nagata extension of the valued local domain $R$ and a polynomial $h(X)\in R[X]$, there exist $a\in R_\nu$ and $e\in\N,\ 0\leq e\leq {\rm deg}h(X)$ such that for all $i$ large enough we have the equality ${\rm in}_\nu h(\sigma_i)={\rm in}_\nu(a\delta_i^e)$.
\end{proposition}
\begin{proof}Considering $h(X)$ as a polynomial in $R_\nu[X]$ we see that it suffices to prove the result when one of the coefficients of $h(X)$ is equal to one, so we assume this. We use the notation of the proof of Proposition~\ref{computationNagata}. In particular, let $Q(X)\in R_\nu[X]$ denote the lifting of the minimal polynomial of $\bar\sigma_\infty$ over $L$ and let us write the $Q(X)$-adic expansion of $h(X)$ as $$h(X)=A_r(X)Q(X)^r+A_{r-1}(X)Q(X)^{r-1}+\cdots +A_0(X)\ \ {\rm with}\ {\rm deg}A_j(X)<{\rm deg}Q(X).$$ 
We have a similar expression after passing to the quotient by $m_{\Psi_F}$. Let $J$ be the non--empty set consisting of those $j$ such that $0\leq j\leq r$ and $\widebar A_j(X)\neq0$. For all $j\in J$, the condition ${\rm deg}\widebar A_j(X)<{\rm deg}\widebar Q(X)$ implies that there exists $\varphi_j\in\Psi_F$ such that $\nu(A_j(\sigma_i))=\varphi_j$ for all $i$ large enough (see Lemma~\ref{quot}), and thus we can find $b_j\in R_\nu$ such that ${\rm in}_\nu A_j(\sigma_i)={\rm in}_\nu b_j$ for large $i$ (see Corollary~\ref{approx}). Evaluating in $\sigma_i$ the identity \eqref{B} of the proof of Proposition~\ref{computationNagata}, we get ${\rm in}_\nu Q(\sigma_i)={\rm in}_\nu(-a_{n-1}G_1(0)^{-1}\delta_i)$ for all $i\geq1$. Since $\nu(A_j(\sigma_i)Q(\sigma_i)^j)\notin\Psi_F$ if $j\notin J$, the result follows from Lemma~\ref{rat} and Proposition~\ref{OK} with $\gamma_i=\nu(Q(\sigma_i))$, and for all $j\in J$, $\beta_j=\varphi_j$ and $t_j=j$. The case where $(\nu(h(\sigma_i)))_{i\geq 1}$ is eventually constant corresponds to $e=0$ and then we have $a\in R$.\end{proof}
\begin{remark}
Since $R_\nu$ and its henselization $R_{\tilde\nu}$ have the same residue field, another way to state that $\nu$ and $\tilde\nu$ have the same value group is to say (see \cite[Proposition 4.1]{T}) that the natural graded injection ${\rm gr}_\nu R_\nu\hookrightarrow {\rm gr}_{\tilde\nu}R_{\tilde\nu}$ is an equality. It is therefore not surprising that the initial form of $h_k(\sigma_\infty)$ in Proposition~\ref{asym} appears as the initial form of an element of $R_\nu$.
\end{remark}
\begin{proof}[Proof of Theorem~\ref{ap}] 
Let us first assume that $z$ lies in $R_{\tilde\nu}$. Using the fact that $R_\nu$ and $R_{\tilde\nu}$ have the same residue field, by removing an element of $R_\nu$ of value zero, we can exclude the case where $\tilde\nu(z)=0$ and assume $\tilde\nu(z)>0$. Then $z$ lies in the maximal ideal of a Nagata extension $R_0[\si]_*\subset R_{\tilde\nu}$ of a normal local domain $R_0\subset R_\nu$ essentially of finite type over the prime ring. We call $\nu_0$ the restriction of $\nu$ to the fraction field $K_0$ of $R_0$ and $\Phi_0$ its value group. Let $F(X)\in R_0[X]$ be an irreducible Nagata polynomial defining $R_0[\si]_*$ and let $(\delta_i)_{i\in\N}$ be its Newton sequence of values. After Proposition~\ref{nu-ext}, since $R_0$ is noetherian and $\Phi_0$ of finite rank, we may assume that $\si$ is a limit of the pseudo--convergent sequence $(\sigma_i)_{i\geq1}$ associated to $F[X]$ in such a way that for any polynomial $P(X)\in R_0[X]$ with $0\leq\deg P(X)<\deg F(X)$, we have that ${\rm in}_{\tilde\nu}P(\si)={\rm in}_{\nu}P(\sigma_i)$ for large $i$.

We write $z=\frac{h(\si)}{q(\si)}$ in $R_0[\si]_*$, where $h(X),q(X)\in R_0[X]$ are polynomials of degree less than $\deg F(X)$, $h(0)\in m_{R_0}$, and $q(0)\notin m_{R_0}$. The polynomial $$H(X)=h(\si)q(X)-h(X)q(\si)\in R_0[\si]_*[X]$$ satisfies the equation $H(\si)=0$. It is of positive degree since otherwise it would be identically zero and $\frac{h(X)}{q(X)}$ would be constant. By the same argument as in the proof of Proposition~\ref{asym}, the $\tilde\nu$-value of $H(\sigma_i)$ is, for large $i$, of the form $\tilde\nu(H_k(\si)\delta_i^k)$ with $k\geq 1$. 
In addition, for all $i$ large enough we have $\nu(q(\sigma_i))=\tilde\nu(q(\si))=0$ and the $\tilde\nu$-value of $H(\sigma_i)=h(\si)q(\sigma_i)-h(\sigma_i)q(\si)$ coincides with that of $\frac{h(\si)}{q(\si)}-\frac{h(\sigma_i)}{q(\sigma_i)}$. Thus, by Lemma~\ref{cofinal} there exists $i_0\geq1$ such that $$\set{\tilde\nu\left(\frac{h(\si)}{q(\si)}-\frac{h(\sigma_i)}{q(\sigma_i)}\right)\,\vert\,i\geq i_0} \text{ is cofinal in } \varphi+\Psi,$$ where $\varphi=\tilde\nu(H_k(\si))\in\Phi_{\geq0}$ and $\Psi$ is the smallest convex subgroup of $\Phi$ containing all the $\nu(\delta_i)$. Note that since the $\delta_i$ associated to $F(X)$ are in $R_0$, the smallest convex subgroup of $\Phi_0$ containing the $\nu(\delta_i)$ is the intersection with $\Phi_0$ of the convex subgroup $\Psi_F$ of $\Phi$ associated to the polynomial $F(X)$ seen as a Nagata polynomial in $R_\nu[X]$ and is cofinal in it by Lemma~\ref{cofinal}. Next we prove that $\varphi+\Psi$ is cofinal in $\tilde\nu(z-K)$.

Let us first verify the inclusion $\varphi+\Psi\subset\tilde\nu(z-K)$. Given $\psi\in\Psi$ and an element $c\in K$ such that $\nu(c)=\varphi+\psi$, the cofinality we verified above implies that there exists $i\geq i_0$ such that $\tilde\nu\left(\frac{h(\si)}{q(\si)}-\frac{h(\sigma_i)}{q(\sigma_i)}\right)>\varphi+\psi$. Therefore we also have $\varphi+\psi=\tilde\nu\left(\frac{h(\si)}{q(\si)}-\frac{h(\sigma_i)}{q(\sigma_i)}+c\right)$, which is an element of $\tilde\nu(z-K)$.
 
Now let us prove that for any $c\in K$ there exists $\psi\in\Psi$ such that $\varphi+\psi>\tilde\nu(z-c)$. 
We may assume that $c\in R_\nu$ since otherwise $\tilde\nu(z-c)=\tilde\nu(c)<0$ and the result is clear. Then, enlarging the local ring $R_0$ if necesssary, we can also assume that $c\in R_0$. So let us consider the polynomial $h_{(c)}(X)=h(X)-cq(X)\in R_0[X]$ and write $z-c$ as $\frac{h_{(c)}(\si)}{q(\si)}$ in $R_0[\si]_*$. Since both $h_{(c)}(X)$ and $q(X)$ are nonzero polynomials of degree less than $\deg F(X)$, we have by Corollary~\ref{approx} or Proposition \ref{ae} the equality $${\rm in}_{\tilde\nu}(h_{(c)}(\si)q(\sigma_i))={\rm in}_{\tilde\nu} (h_{(c)}(\sigma_i)q(\si))\text{ for all }i\text{ large enough},$$ which implies the inequality $\tilde\nu\left(\frac{h_{(c)}(\si)}{q(\si)}\right)<\tilde\nu\left(\frac{h_{(c)}(\si)}{q(\si)}-\frac{h_{(c)}(\sigma_i)}{q(\sigma_i)}\right)$. But $\frac{h_{(c)}(\si)}{q(\si)}=z-c$ and $\frac{h_{(c)}(\si)}{q(\si)}-\frac{h_{(c)}(\sigma_i)}{q(\sigma_i)}=\frac{h(\si)}{q(\si)}-\frac{h(\sigma_i)}{q(\sigma_i)}$ for large $i$, so that we can find $i_1\geq i_0$ such that $$\tilde\nu(z-c)<\tilde\nu \left(\frac{h(\si)}{q(\si)}-\frac{h(\sigma_i)}{q(\sigma_i)}\right)=\varphi+\nu(\delta_i^k) \text{ for all } i\geq i_1.$$ This inequality gives the result we want in this case. \par\noindent
If $z\in K^h\setminus  R_{\tilde\nu}$, using the fact that the value groups of $\nu$ and $\tilde\nu$ are the same, we choose $d\in m_\nu$ such that $dz\in m_{\tilde\nu}$, apply to $dz$ the argument we have just seen and use the fact that $\tilde\nu(dz-K)=\tilde\nu(z-K)+\nu(d)$. Replacing the element $\varphi$ associated to $dz$ as above by $\varphi-\nu(d)$ gives the result.\end{proof}
\begin{remark} This gives a more precise form to the result of Kaplansky in \cite[Theorem 1]{K}, stating the since $K^h$ is an immediate extension of $K$, every element of $K^h\setminus K$, such as our $z=\frac{h(\si)}{q(\si)}$, is a limit of a pseudo-convergent sequence of elements of $K$ without a limit in $K$, in this case $(\frac{h(\sigma_i)}{q(\sigma_i)})_{i\geq 1}$. Theorem \ref{ap} gives us precise information about the values of the differences $\frac{h(\si)}{q(\si)}-\frac{h(\sigma_i)}{q(\sigma_i)}$, in addition to the fact that they are cofinal in $\tilde\nu (z-K)$.
\end{remark}

\section{Etale type and the henselian property}\label{sec4}
In this section we relate in greater generality Nagata polynomials with certain pseudo--convergent sequences and obtain a valuative characterization of the henselian property.\par

After Ostrowski and Kaplansky, one says that a pseudo--convergent sequence $(y_\tau)_{\tau\in T}$ of elements of a valued field $(K,\nu)$ is of \textit{algebraic type} if there exist polynomials $h(X)\in K[X]$ such that $(\nu(h(y_\tau)))_{\tau\in T}$ is not eventually constant. We propose the following, where as usual $\tau+1$ designates the successor of $\tau$ in the well ordered set $T$:\par\noindent
\begin{definition}\label{eta}Let $\nu$ be a valuation centered in a local domain $R$. A pseudo--convergent sequence $(y_\tau)_{\tau\in T}$ of elements of the maximal ideal $m_R$ of $R$ is \textit{of \'etale type} if there exist polynomials $h(X)\in R[X]$ such that one has the equality $\nu(h(y_\tau))=\nu(y_{\tau +1}-y_\tau)$ for $\tau\geq \tau_0\in T$, where $\tau_0$ may depend on the polynomial $h(X)$.
\end{definition}

Note that if the values of the $y_\tau$ are not eventually constant, then $\nu(y_{\tau+1})>\nu(y_\tau)$ for large $\tau$ and then $h(X)=X-a$ with $a\in R$ such that $\nu(a)>\nu(y_\tau)$ for all $\tau\in T$ is such a polynomial. The element $a\in R$ is a limit of $(y_\tau)_{\tau\in T}$. By the argument given in Lemma \ref{cofinal}, a pseudo-convergent sequence such that the $\nu(y_\tau)$ are not eventually constant is tested as being of \'etale type by a linear polynomial $X-a$ if $\nu(a)$ is not in the smallest convex subgroup containing the $\nu(y_\tau)$. \par

It is a classical result (see \cite[Theorem 4]{K}, \cite[Theorem 8.19]{Ku3}) that a valued field is maximal (has no non trivial immediate extension, which means a  valued extension with the same value group and the same residue field) if and only if all pseudo--convergent sequences in the field have a limit in the field. Since henselization is an immediate extension by Corollary~\ref{V}, maximal valued fields have a henselian valuation ring. We give a somewhat more precise and more general result in the following valuative criterion for the henselian property:\par\noindent

\begin{proposition}\label{valcrithen}
Let $R$ be a local domain with maximal ideal $m_R$, and let $\nu$ a valuation of finite rank centered in $R$. The local domain $R$ is henselian if and only if every pseudo--convergent sequence of elements of $m_R$ which is of \'etale type has a limit in $m_R$.
\end{proposition}

\begin{proof}
A local ring in which every Nagata polynomial has a root in the maximal ideal is henselian; see \cite[Lemma, p.94]{L}. This and the fact that the henselization is the inductive limit of Nagata extensions imply that the local domain $R$ is henselian if and only if every Nagata polynomial in $R[X]$ has a root in $m_R$.

Let us assume that every pseudo--convergent sequence of elements of $m_R$ which is of \'etale type has a limit in $m_R$. We proceed by contradiction and suppose that there exists a Nagata polynomial $F(X)\in R[X]$ of degree $n$ such that $F(a)\neq0$ for all $a\in m_R$. By Lemma \ref{good}, we may assume that $X$ is a good variable. Then the polynomial $F(X)$ comes with the pseudo--convergent sequence of partial sums $(\sigma_i)_{i\geq1}$ and the associated convex subgroup $\Psi_F$ of the value group of $\nu$, which equals the intrinsic convex subgroup $\Psi$. Our assumption implies that $(\sigma_i)_{i\geq1}$ has a limit $y\in m_R$. We have $F(y)=F(\sigma_i)+\sum_{m=1}^n F_m(\sigma_i)(y-\sigma_i)^m$ as in the identity \eqref{eq:1} of Section~\ref{sec3}, and since $\nu(y-\sigma_i)\geq \nu(\delta_i)$ and $\nu(F(\sigma_i))=\nu(\delta_i)$, we get that $\nu(F(y))\geq\nu(\delta_i)$ for all $i$. Since $\nu(F(y))\in\Psi=\Psi_F$ and the $\nu(\delta_i)$ are strictly increasing and cofinal in $\Psi_F$, this gives us a contradiction.

Let us now assume that $R$ is henselian and let $(y_\tau)_{\tau \in T}$ be a pseudo--convergent sequence of \'etale type of elements of $m_R$. Let $h(X)\in R[X]$ be a polynomial verifying the equality $\nu(h(y_\tau))=\nu(y_{\tau +1}-y_\tau)$ for large $\tau$. Suppose that $h(X)$ is, up to multiplication by an invertible element of $R$, a Nagata polynomial. Then $h(X)$ has a root $y\in m_R$ since $R$ is henselian. We can write $h(X)=(X-y)G(X)$ in $R[X]$ and by Lemma~\ref{fac}, since $X-y$ is a Nagata polynomial, we have that $G(X)\notin(m_R,X)$ so that $\nu(h(y_\tau))=\nu(y_{\tau +1}-y_\tau)=\nu(y-y_\tau)$ for large $\tau$, which shows that $y$ is a limit of $(y_\tau)_{\tau \in T}$. Therefore, in order to end the proof, it is sufficient to show that $h(X)$ is a Nagata polynomial up to multiplication by a unit of $R$.

Writing $h(X)=b_0X^s+\cdots +b_{s-1}X+b_s=b_0\prod_{j=1}^s (X-r_j)$ in a splitting field of $h(X)$, after extension of $\nu$ we see that our assumption immediately implies that $\nu(b_0)=0$ and exactly one of the $r_j$ is a limit of $(y_\tau)_{\tau \in T}$ while the values of the other $y_\tau-r_j$ are zero. So after multiplication by an invertible element of $R$ we may assume that $h(X)$ is unitary. We have the identity $h(y_{\tau +1})-h(y_\tau)=\sum _{m=1}^sh_m(y_\tau)(y_{\tau +1}-y_\tau)^m$ as in the identity \eqref{eq:2} of Section~\ref{sec3}. Since $\nu(h(y_{\tau+1}))>\nu(h(y_\tau))$ at least for large $\tau$, the previous equality then implies the equality 
\begin{equation*}
\begin{split}
{\rm in}_\nu h(y_{\tau}) & = -{\rm in}_\nu\big(\sum _{m=1}^sh_m(y_{\tau})(y_{\tau+1}-y_\tau)^m\big)\\
& = -{\rm in}_\nu\big((y_{\tau+1}-y_\tau)(h_1(y_\tau)+\sum _{m=2}^sh_m(y_{\tau})(y_{\tau+1}-y_\tau)^{m-1})\big),
\end{split}
\end{equation*}
so that $h(y_\tau)$ can have the same value as $y_{\tau+1}-y_\tau$ only if $h_1(y_\tau)$ is invertible, which in turn implies, since $y_\tau\in m_R$, that the coefficient $b_{s-1}$ of $X$ in $h(X)$ is invertible in $R$. Finally, we get that $h(0)=b_s\in m_R$ since $y_\tau\in m_R$ for all $\tau$. This shows that $h(X)$ is a Nagata polynomial.\end{proof}

\bibliographystyle{alpha}

\end{document}